\documentclass{article}

\usepackage[british]{babel}
\usepackage[applemac]{inputenc}
\usepackage{amsmath}
\usepackage{mathtools}
\usepackage{amssymb}
\usepackage{amsthm}
\usepackage{dsfont}
\usepackage{xcolor}
\usepackage{enumitem}
\usepackage{graphicx}
\usepackage{hyperref}
\usepackage{cleveref}
\usepackage{geometry}
\usepackage[normalem]{ulem}

\newtheorem{theorem}{Theorem}[section]
\newtheorem{conj}[theorem]{Conjecture}

\newtheorem{lemma}[theorem]{Lemma}

\theoremstyle{remark}
\newtheorem{rem}[theorem]{Remark}

\theoremstyle{definition}
\newtheorem{definition}[theorem]{Definition}

\newcounter{saveenumerate}
\makeatletter
\newcommand{\enumeratext}[1]{%
\setcounter{saveenumerate}{\value{enum\romannumeral\the\@enumdepth}}
\end{enumerate}
#1
\begin{enumerate}
\setcounter{enum\romannumeral\the\@enumdepth}{\value{saveenumerate}}%
}
\makeatother

\DeclarePairedDelimiterX{\norm}[1]{\lVert}{\rVert}{#1}

\newcommand{\N}{\ensuremath{\mathbb{N}}}

\newcommand{\R}{\ensuremath{\mathbb{R}}}

\newcommand{\tr}{\text{tr}}
\newcommand{\rank}[1]{\text{\textnormal{rk}}(#1)}
\newcommand{\trace}[1]{\text{\textnormal{tr}}(#1)}
\newcommand{\bangle}[1]{\left\langle #1 \right\rangle}

\newcommand{\wellsee}{k}
\newcommand{\wellseeminusone}{k-1}
\newcommand{\wellseeplusone}{k+1}
\newcommand{\finalconstant}{1.92}
\newcommand{\finalconstantplusone}[0]{1.93}
\newcommand{\finalconstantminusone}{0.92}

\newcommand{\numberofparts}{11}
\newcommand{\numberofpartsminusone}{10}
\newcommand{\twotimesnumberofpartsminusone}{18}
\newcommand{\twotimesnumberofparts}{20}

\let\dim\relax \newcommand{\dim}[0]{n}
\newcommand{\gammma}[0]{\gamma}
\newcommand{\dleta}[0]{\sigma}
\newcommand{\average}[0]{d}
\newcommand{\betta}[0]{\beta}
\newcommand{\func}[0]{f}
\newcommand{\edges}[0]{h}
\newcommand{\allonesvector}[0]{\mathds{1}}
\newcommand{\inprod}[2]{\bangle{#1, #2}}
\newcommand{\pmalpha}{\{-\alpha,\alpha \}}
\newcommand{\equimax}[1]{N(#1)}
\newcommand{\equimaxangle}[2]{N_{#1}(#2)}

\newcommand{\spherical}[0]{L}
\newcommand{\code}[0]{C}

\newcommand{\polygram}[0]{Q}
\newcommand{\commonangle}[0]{\theta}
\let\Pr\relax \DeclareMathOperator{\Pr}{Pr\,}

\newcommand{\Span}{\text{span}}

\date{}
\title{Equiangular Lines and Spherical Codes in Euclidean Space}
\author{
Igor Balla \thanks{Department of Mathematics, ETH, 8092 Zurich. igor.balla@math.ethz.ch.}
\and
Felix Dr\"axler \thanks{Department of Mathematics, ETH, 8092 Zurich. felix.draexler@math.ethz.ch.}
\and
Peter Keevash \thanks{Mathematical Institute, University of Oxford, Oxford, UK. keevash@maths.ox.ac.uk. Research supported in part by ERC Consolidator Grant 647678.}
\and
Benny Sudakov \thanks{Department of Mathematics, ETH, 8092 Zurich.
benjamin.sudakov@math.ethz.ch. 
Research supported in part by SNSF grant 200021-149111.}
}

\begin{document}
\maketitle
\abstract{A family of lines through the origin in Euclidean space is called equiangular if any pair of lines defines the same angle. The problem of estimating the maximum cardinality of such a family in $\mathbb{R}^\dim$ was extensively studied for the last 70 years. Motivated by a question of Lemmens and Seidel from 1973, in this paper we prove that for every fixed angle $\commonangle$ and sufficiently large $\dim$ there are at most $2\dim-2$ lines in $\mathbb{R}^\dim$ with common angle $\commonangle$. Moreover, this bound is achieved if and only if $\commonangle = \arccos \frac{1}{3}$. Indeed, we show that for all $\theta \neq \arccos{\frac{1}{3}}$ and and sufficiently large $n$, the number of equiangular lines is at most $1.93 n$. We also show that for any set of $k$ fixed angles, one can find at most $O(\dim^k)$ lines in $\mathbb{R}^\dim$ having these angles. This bound, conjectured by Bukh, substantially improves the estimate of Delsarte, Goethals and Seidel from 1975. Various extensions of these results to the more general setting of spherical codes will be discussed as well.}

\section{Introduction}

A set of lines through the origin in $\dim$-dimensional Euclidean space is called \textit{equiangular} if any pair of lines defines the same angle. Equiangular sets of lines appear naturally in various areas of mathematics. In elliptic geometry, they correspond to equilateral sets of points, or, in other words, to regular simplexes. These simplexes were first studied 70 years ago \cite{H48}, since the existence of large regular simplexes leads to high congruence orders of elliptic spaces, see \cite{B53, HS47, vLS66}. In frame theory, so-called Grassmannian frames ``are characterised by the property that the frame elements have minimal cross-correlation among a given class of frames'' \cite{HS03}. It turns out that optimal Grassmannian frames are equiangular; hence searching for equiangular sets of lines is closely related to searching for optimal Grassmannian frames, see \cite{HS03}. In the theory of polytopes, the convex hull of the points of intersection of an equiangular set of lines with the unit sphere is a spherical polytope of some kind of regularity, see \cite{C73}. 

It is therefore a natural question to determine the maximum cardinality $\equimax{\dim}$ of an equiangular set of lines in $\mathbb{R}^\dim$. This is also considered to be one of the founding problems of algebraic graph theory, see e.g. \cite[p. 249]{GR01}. While it is easy to see that $\equimax{2} \le 3$ and that the three diagonals of a regular hexagon achieve this bound, matters already become more difficult in $3$ dimensions. This problem was first studied by Haantjes \cite{H48} in 1948, who showed that $\equimax{3} = \equimax{4} = 6$ and that an optimal configuration in $3$ (and $4$) dimensions is given by the $6$ diagonals of a convex regular icosahedron. In 1966, van Lint and Seidel \cite{vLS66} formally posed the problem of determining $\equimax{\dim}$ for all positive integers $\dim$ and furthermore showed that $\equimax{5} = 10$, $\equimax{6} = 16$ and $\equimax{7} \ge 28$.

A general upper bound of $\tbinom{\dim+1}{2}$ on $\equimax{\dim}$ was established by Gerzon (see \cite{LS73}). Let us outline his proof. Given an equiangular set of $m$ lines in $\mathbb{R}^\dim$, one can choose a unit vector $x_i$ along the $i$th line to obtain vectors $x_1, \dots, x_m$ satisfying $\inprod{x_i}{x_j} \in \pmalpha$ for $i \ne j$. Consider the family of outer products $x_ix_i^\intercal$; they live in the $\tbinom{\dim+1}{2}$-dimensional space of symmetric $\dim \times \dim$ matrices, equipped with the inner product $\inprod{A}{B} = \trace{A^\intercal B}$. It is a routine calculation to verify that $\inprod{x_ix_i^\intercal\negthinspace}{x_jx_j^\intercal} = \inprod{x_i}{x_j}^2$, which equals $\alpha^2$ if $i\ne j$ and $1$ otherwise. This family of matrices is therefore linearly independent, which implies $m \le \tbinom{\dim+1}{2}$.

In dimensions $2$ and $3$ this gives upper bounds of $3$ and $6$, respectively, matching the actual maxima. In $\mathbb{R}^7$, the above bound shows $N(7) \le 28$. This can be achieved by considering the set of all $28$ permutations of the vector $(1,1,1,1,1,1,-3,-3)$, see \cite{vLS66,S67}. Indeed, one can verify that the dot product of any two distinct such vectors equals either $-8$ or $8$, so that after normalising the vectors to unit length this constitutes an equiangular set of lines. Since the sum of the coordinates of each such vector is $0$, they all live in the same $7$-dimensional subspace. It is also known that there is an equiangular set of 276 lines in $\R^{23}$, see e.g. \cite{LS73}, which again matches Gerzon's bound. Strikingly, these four examples are the only known ones to match his bound \cite{BY14}. In fact, for a long time it was even an open problem to determine whether $\dim^2$ is the correct order of magnitude. In 2000, de Caen \cite{dC00} constructed a set of $2(\dim+1)^2/9$ equiangular lines in $\mathbb{R}^\dim$ for all $\dim$ of the form $3\cdot 2^{2t-1}-1$. Subsequently, several other constructions of the same order were found \cite{BY14, GKMS16, JW15}. For further progress on finding upper and lower bounds on $\equimax{\dim}$ see e.g. \cite{BY14} and its references.

Interestingly, all the above examples of size $\Theta(\dim^2)$ have a common angle on the order of $\arccos (1/\sqrt{\dim})$. On the other hand, all known  construction of equiangular lines  with a fixed common angle have much smaller size. It is therefore natural to consider the maximum number $\equimaxangle{\alpha}{\dim}$ of equiangular lines in $\mathbb{R}^\dim$ with common angle $\arccos \alpha$, where $\alpha$ does not depend on dimension. This question was first raised by Lemmens and Seidel \cite{LS73} in 1973, who showed that for sufficiently large $\dim$, $\equimaxangle{1/3}{\dim} = 2\dim - 2$ and also conjectured that $\equimaxangle{1/5}{\dim}$ equals $\lfloor 3(\dim-1)/2 \rfloor$. This conjecture was later confirmed by Neumaier \cite{N89}, see also \cite{GKMS16} for more details. Interest in the case where $1/\alpha$ is an odd integer was due to a general result of Neumann \cite[p. 498]{LS73}, who proved that if $\equimaxangle{\alpha}{\dim} \ge 2\dim$, then $1/\alpha$ is an odd integer.

Despite active research on this problem, for many years these were the best results known. Recently, Bukh \cite{B15} made important progress by showing that $\equimaxangle{\alpha}{\dim} \le c_{\alpha}\dim$, where $c_\alpha = 2^{O(1/\alpha^2)}$ is a large constant only depending on $\alpha$. Our first main result completely resolves the question of maximising $\equimaxangle{\alpha}{\dim}$ over constant $\alpha$. We show that for sufficiently large $\dim$, $\equimaxangle{\alpha}{\dim}$ is maximised at $\alpha = \frac{1}{3}$.

\begin{theorem}
\label{t equiangular}
  Fix $\alpha \in (0,1)$. For $\dim$ sufficiently large relative to $\alpha$, the maximum number of equiangular lines in $\mathbb{R}^{\dim}$ with angle $\arccos\alpha$ is exactly $2\dim-2$ if $\alpha = \frac{1}{3}$ and at most $\finalconstantplusone \dim$ otherwise.
\end{theorem}

A more general setting than that of equiangular lines is the framework of \textit{spherical $\spherical$-codes}, introduced in a seminal paper by Delsarte, Goethals and Seidel \cite{DGS77} in 1977 and extensively studied since.

\begin{definition}
  Let $\spherical$ be a subset of the interval $[-1,1)$. A finite non-empty set $\code$ of unit vectors in Euclidean space $\mathbb{R}^\dim$ is called a \textnormal{spherical $\spherical$-code}, or for short an \textnormal{$\spherical$-code}, if $\inprod{x}{y} \in \spherical$ for any pair of distinct vectors $x,y$ in $\code$.
\end{definition}

Note that if $\spherical = \pmalpha$, then an $\spherical$-code corresponds to a set of equiangular lines with common angle $\arccos \alpha$, where $\alpha \in [0,1)$. For $\spherical = [-1,\beta]$, finding the maximum \mbox{cardinality} of an $\spherical$-code is equivalent to the classical problem of finding non-overlapping spherical caps of angular radius $\tfrac{1}{2}\arccos \beta$; for $\beta \le 0$ exact formulae were obtained by Rankin \cite{R55}. Generalising Gerzon's result, Delsarte, Goethals and Seidel \cite{DGS75} obtained bounds on the cardinality of sets of lines having a prescribed number of angles. They proved that, for $\spherical = \{ -\alpha_1, \dots, -\alpha_k, \alpha_1, \dots, \alpha_k\}$ and $\alpha_1, \dots, \alpha_k \in [0,1)$, spherical $\spherical$-codes have size at most $O(\dim^{2k})$. They subsequently extended this result to an upper bound of $O(\dim^s)$ on the size of an $\spherical$-code when $\spherical$ has cardinality $s$, see \cite{DGS77}. A short proof of this estimate based on the polynomial method is due to Koornwinder \cite{K76}.

Bukh \cite{B15} observed that, in some sense, the negative values of $\spherical$ pose less of a constraint on the size of $\spherical$-codes than the positive ones, as long as they are separated away from $0$. Specifically, he proved that for $\spherical = [-1,-\beta] \cup \{\alpha\}$, where $\beta \in (0,1)$ is fixed, the size of any $\spherical$-code is at most linear in the dimension. Motivated by the above-mentioned work of Delsarte, Goethals and Seidel \cite{DGS75} he made the following conjecture.

\begin{conj} \label{c beta fixed}
  Let $\beta \in (0,1)$ be fixed and let $\alpha_1, \dots, \alpha_k$ be any $k$ reals. Then any spherical $[-1,-\beta] \cup \{\alpha_1, \dots, \alpha_k\}$-code in $\mathbb{R}^\dim$ has size at most $c_{\beta,k}\dim^k$ for some constant $c_{\beta,k}$ depending only on $\beta$ and $k$.
\end{conj}

We verify this conjecture in the following strong form.

\begin{theorem}
\label{t spherical}
  Let $L = [-1, -\beta] \cup \{ \alpha_1, \ldots, \alpha_k \}$ for some fixed $\beta \in (0,1]$. Then there exists a constant $c_{\beta, k}$ such that any spherical $L$-code in $\R^n$ has size at most $c_{\beta, k} n^k$. Moreover, if $0 \leq \alpha_1 < \ldots <\alpha_k < 1$ are also fixed then such a code has size at most
  \begin{equation*}
    2^k(k-1)!\Bigl( 1 + \frac{\alpha_1}{\beta} \Bigr)\dim^k + o(\dim^k).
  \end{equation*}
\end{theorem}

In particular, if $\alpha_1, \dots, \alpha_k$ are fixed this substantially improves the aforementioned bound of Delsarte, Goethals and Seidel \cite{DGS75, DGS77} from $O(\dim^{2k})$ to $O(\dim^k)$. We furthermore show that the second statement of \Cref{t spherical} is tight up to a constant factor.

\begin{theorem}
  \label{t construction}
  Let $n, k, r$ be positive integers and $ \alpha_1 \in (0,1)$ with $k$ and $\alpha_1$ being fixed and $r \leq \sqrt{n}$. Then there exist $\alpha_2, \ldots, \alpha_k$, $\beta = \alpha_1 / r - O(\sqrt{\log(n) / n})$ and a spherical $L$-code of size $(1 + r) \binom{n}{k}$ in $\R^{n + r}$
  with $L = [-1, -\beta] \cup \{ \alpha_1, \ldots, \alpha_k \}$.
\end{theorem}

This also resolves another question of Bukh, who asked whether the maximum size of a spherical $[-1,0)\cup \{\alpha\}$-code in $\mathbb{R}^\dim$ is linear in $\dim$. By taking $\beta$ to be say $\log(n) / \sqrt{n}$, our construction demonstrates that this is not the case.

The rest of this paper is organised as follows. In \Cref{s equiangular lines} we give a construction of an equiangular set of $2\dim -2$ lines in $\mathbb{R}^\dim$ and prove \Cref{t equiangular}. In \Cref{s spherical codes} we prove a special case of \Cref{t spherical}, namely the case $k=1$. We provide the construction which shows that our bounds are asymptotically tight in Section 4. In \Cref{s multi-angular spherical codes}, we prove \Cref{t spherical}. The last section of the paper contains some concluding remarks and open problems.

\subsection*{Notation}
We will always assume that the dimension $n \rightarrow \infty$ and write $f = o(1)$, respectively $f = O(1)$ to mean $f(n) \rightarrow 0$ as $n \rightarrow \infty$, respectively $f(n) \leq C$ for some constant $C$ and $n$ sufficiently large. We will say $\gamma$ is fixed to mean that it does not depend on $n$.

Let $C = \{v_1, \ldots, v_m\}$ be a spherical $L$-code in $\R^n$. We define $M_C$ to be the associated Gram matrix given by $(M_C)_{i,j} = \inprod{v_i}{v_j}$. We also define an associated complete edge-labelled graph $G_C$ as follows: let $C$ be its vertex set and for any distinct $u,v \in C$, we give the edge $uv$ the value $\gamma$ iff $\inprod{u}{v} = \gamma$. We also say that $uv$ is a $\gamma$-edge and for brevity, we sometimes refer to $\gamma$ as the ``angle'' between $u$ and $v$, instead of the ``cosine of the angle''. For $\beta > 0$, we slightly abuse our notation and say that $uv$ is a $\beta$-edge if $\inprod{u}{v} \leq - \beta$. We call a subset $S \subset G_C$ a $\gamma$-clique if $uv$ is a $\gamma$-edge for all distinct $u,v \in S$. For any $x \in G_C$ we define the $\gamma$-neighbourhood of $x$ to be $N_\gamma(x) = \{ y \in G_C : xy \text{ is a } \gamma \text{-edge} \}$. Furthermore, we define the $\gamma$-degree $d_\gamma(x) = | N_\gamma(x) |$ and the maximum $\gamma$-degree $\Delta_\gamma = \max_{x \in G}{d_\gamma(x)}$. 

We denote the identity matrix by $I$  and denote the all 1's matrix by $J$, where the size of the matrices is always clear from context. Let $Y$ be a set of vectors in $\R^n$. We define $\Span(Y)$ to be the subspace spanned by the vectors of $Y$ and for a subspace $U$, define $U^{\perp} = \{x \in \R^n : \inprod{x}{y} = 0 \text{ for all } y \in U\}$ to be the orthogonal complement. For all $x \in \R^n$ define $p_Y(x)$ to be the normalised (i.e. unit length) projection of $x$ onto the orthogonal complement of $\Span(Y)$, provided that the projection is nonzero. That is, if we write $x = u+v$ for $u \in \Span(Y)^\perp$, $v \in \Span(Y)$ and $u \neq 0$, then $p_Y(x) = u/\norm{u}$. More generally, for a set of vectors $S$ we write $p_Y(S) = \{ p_Y(x) : x \in S\}$.

\section{Equiangular lines}
\label{s equiangular lines}

Suppose that we are given a set of equiangular lines in $\dim$-dimensional Euclidean space $\mathbb{R}^{\dim}$ with common angle $\arccos \alpha$. By identifying each line with a unit vector along this line, we obtain a set of unit vectors with the property that the inner product of any two vectors equals either $\alpha$ or $-\alpha$. As we have already mentioned in the introduction, we will refer to such a set as a $\pmalpha$-code. Given a $\pmalpha$-code $C$, we call an $\alpha$-edge of $G_C$ \textit{positive} and a $-\alpha$-edge \textit{negative}.

Van Lint and Seidel \cite{vLS66} observed that a particular set of equiangular lines corresponds to various $\pmalpha$-codes, depending on which of the two possible vectors we choose along each line. Conversely, this means that we can negate any number of vectors in a $\pmalpha$-code without changing the underlying set of equiangular lines. In the corresponding graph, this means that we can switch all the edges adjacent to some vertex from positive to negative and vice versa.

The proof of \Cref{t equiangular} builds on several key observations. The first is that we can use Ramsey's theorem to find a large positive clique in $G_C$. We then negate some vertices outside of this clique, in order to obtain a particularly advantageous graph, for which we can show that almost all vertices attach to this positive clique entirely via positive edges. We then project this large set onto the orthogonal complement of the positive clique. Next we observe that the resulting graph contains few negative edges, which implies that the diagonal entries of the Gram matrix of the projected vectors are significantly larger in absolute value than all other entries. Combining this with an inequality which bounds the rank of such matrices already gives us a bound of $(2 + o(1))n$. To prove the exact result, we use more carefully the semidefiniteness of the Gram matrix together with some estimates on the largest eigenvalue of a graph.

We finish this discussion by giving an example of an equiangular set of $2\dim-2$ lines with common angle $\arccos \tfrac{1}{3}$ in $\mathbb{R}^\dim$, first given by Lemmens and Seidel \cite{LS73}. This is equivalent to constructing a spherical $\{-\tfrac{1}{3}, \tfrac{1}{3}\}$-code $C$ of size $2 \dim - 2$. For any such code $C$, observe that the Gram matrix $M_C$ is a symmetric, positive-semidefinite $(2n-2) \times (2n-2)$ matrix with $1$s on the diagonal and rank at most $\dim$. Conversely, if $M$ is any matrix satisfying all properties listed, then $M$ is the Gram matrix of a set of $2n-2$ unit vectors in $\mathbb{R}^{\dim}$, see e.g. \cite[Lemma 8.6.1]{GR01}. Thus it suffices to construct such a matrix. To that end, consider the matrix $M$ with $n-1$ blocks on the diagonal, each of the form
\[ 
\left( \begin{matrix}
1 & -\frac{1}{3}\\
-\frac{1}{3} & 1
\end{matrix} \right),
\]
and all other entries $\frac{1}{3}$. Clearly $M$ is a $(2n-2) \times (2n-2)$ symmetric matrix, so we just have to verify that it is positive-semidefinite and of rank $\dim$. To do so, we need to show that $M$ has smallest eigenvalue $0$ with multiplicity $\dim -2$. This is a routine calculation.

\subsection{Orthogonal projections}

Before we can delve into the proof of \Cref{t equiangular}, we will set the ground by providing some necessary lemmas. We start with a well-known upper bound on the size of a negative clique, which will guarantee us a large positive clique using Ramsey's theorem. For later purposes, the lemma is stated in some more generality.

\begin{lemma}
\label{l bounding size of a negative clique}
  Let $0 < \alpha < 1$ and let $C$ be a spherical $[-1,-\alpha]$-code in $\mathbb{R}^{\dim}$. Then $|C| \le \alpha^{-1} + 1$.
\end{lemma}

\begin{proof}
  Let $v = \sum_{x \in C} x$. Then, since every $x \in C$ is a unit vector and $\inprod{x}{x'} \le - \alpha$ for $x \ne x'$,
  \begin{equation*}
    0 \le \norm{v}^2 = \sum_{x \in C} \norm{x}^2 + \sum_{\substack{x,x'\in C \\x\ne x'}} \inprod{x}{x'} \le |C| - \alpha |C|(|C|-1),
  \end{equation*}
  which we can rewrite into the desired upper bound on $|C|$.
\end{proof}

 \begin{rem} \label{r simplex}
 We note that equality in the above lemma occurs only if the vectors of $C$ form a regular simplex.
 \end{rem}

As indicated above, this lemma enables us to find a large positive clique in our graph. The next step is to understand how the remaining vertices attach to this clique. A key tool towards this goal is orthogonal projection. We will first need a lemma that lets us compute the inner product between two vectors in the span of a clique in terms of the inner products between the vectors and the clique. Because we will need it again in a later section, we state it in some generality.

\begin{lemma} \label{l inner product}
Let $-1 \leq \gamma < 1$ and $t \neq -1/\gamma  + 1$. Suppose $Y$ is a spherical $\{\gamma\}$-code of size $t$ and $V$ is the matrix with $Y$ as column vectors. Then for all $v_1, v_2 \in \Span(Y)$ we have
\[ \inprod{v_1}{v_2} = \frac{s_1^{\intercal} s_2 - \left( \frac{\gamma}{1 + \gamma(t - 1)} \right)  s_1^{\intercal} J s_2 }{1 - \gamma},\]
where $s_i = V^\intercal v_i$ for $i = 1, 2$ are the vectors of inner products between $v_i$ and $Y$.
\end{lemma}

\begin{proof}
Let us first prove that $Y$ is linearly independent. Suppose that $\sum_{y \in Y} c_y y = 0$ for some reals $c_y$. Taking the inner product with some $y'\in Y$ gives
  \begin{equation*}
    0 = \sum_{y \in Y} c_y \inprod{y}{y'} =  (1 - \gammma)c_{y'} + \gammma \sum_{y \in Y} c_y.
  \end{equation*}
  Since this equation is true for all $y' \in Y$ and $\gammma \ne 1$, all $c_y$ are identical. Unless they equal $0$, this implies that $1 + (t-1) \gammma = 0$, a contradiction.

By passing to a subspace, we may assume that $Y \subset \R^t$, so that $V^{\intercal}$ is invertible and we have $v_i = (V^\intercal)^{-1} s_i$. Thus
\[ \inprod{v_1}{v_2} = ((V^{\intercal})^{-1} s_1)^{\intercal} (V^{\intercal})^{-1} s_2 = s_1^{\intercal} V^{-1} (V^{\intercal})^{-1} s_2= s_1^{\intercal} (V^{\intercal} V)^{-1} s_2. \]
To obtain the result, we observe that $V^\intercal V$ is the Gram matrix of $Y$, so that  $V^\intercal V = (1 - \gamma) I + \gamma J$ and moreover 
\[ (V^\intercal V)^{-1} = \frac{I - \frac{\gamma}{1 + \gamma(t-1)} J}{1 - \gamma} . \qedhere \]
\end{proof}

The following lemma shows how the angle between two vectors changes under an appropriate projection. Recall that $p_Y(x)$ denotes the normalised projection of $x$ onto the orthogonal complement of $\Span(Y)$.

\begin{lemma}
\label{l angle after projecting}
  Let $-1 < \gamma < 1$ and let $Y\cup \{x_1,x_2\}$ be a set of unit vectors in $\mathbb{R}^{\dim}$ so that all pairwise inner products, except possibly $\inprod{x_1}{x_2}$, equal $\gammma$. Suppose additionally that $Y$ has size $1$ if $\gammma$ is negative. Then $p_Y(x_1)$ and $p_Y(x_2)$ are well-defined and we have
  \begin{equation}
  \label{e new angle after projection}
    \inprod{p_Y(x_1)}{p_Y(x_2)} =
    \frac{\inprod{x_1}{x_2} - \gammma}{1-\gammma} +
    \frac{\gammma(1 - \inprod{x_1}{x_2})}{(1+ \gamma|Y|)(1-\gammma)}.
  \end{equation}
\end{lemma}

\begin{proof}
For $i = 1, 2$, write $x_i = u_i + v_i$ where $v_i \in \Span(Y)$ and $u_i \in \Span(Y)^{\perp}$. Let $V$ be the matrix with $Y$ as columns and observe that $s_i = V^\intercal v_i = V^\intercal(x_i - u_i) = (\gamma, \ldots, \gamma)^\intercal$. Let $t = |Y|$ and observe that $t \neq -1/\gamma + 1$ and $t \neq -1 / \gamma$, so we can apply \Cref{l inner product} to obtain
\begin{equation*}
   \inprod{v_1}{v_2} = \inprod{v_1}{v_1} = \inprod{v_2}{v_2} = \frac{t \gamma^2 - \frac{\gamma}{1 + \gamma(t-1)} (t\gamma)^2}{1 - \gamma} = \frac{t \gamma^2}{1 + \gamma(t-1)} < 1.
\end{equation*}
Thus using the fact that $x_i$ is a unit vector and $u_i,v_i$ are orthogonal, we have $||u_i||^2 = 1 - ||v_i||^2 > 0$ for $i = 1, 2$. Since $p_Y(x_i) = u_i / ||u_i||$, we can finish the proof by computing
\begin{equation*}
\begin{split}
    \frac{\inprod{u_1}{u_2}}{\norm{u_1}\norm{u_2}} 
    = \frac{\inprod{x_1}{x_2} - \inprod{v_1}{v_2}}{\sqrt{1 - ||v_1||^2}\sqrt{1 - ||v_2||^2}} 
    = \frac{\inprod{x_1}{x_2} - \frac{t \gamma^2}{1 + \gamma(t-1)}}{1 -  \frac{t \gamma^2}{1 + \gamma(t-1)}} 
    = \frac{\inprod{x_1}{x_2} - \gammma}{1-\gammma} +
    \frac{\gammma(1 - \inprod{x_1}{x_2})}{(1+ \gamma t)(1-\gammma)}.
  \end{split}
  \end{equation*}
\end{proof}

\begin{rem}
\label{r angles go down after projecting}
  Note that when the conditions of the lemma are met we have $\inprod{p_Y(x_1)}{p_Y(x_2)} \le \inprod{x_1}{x_2}$, which in particular implies that $p_Y(x_1) \neq p_Y(x_2)$ when $x_1 \neq x_2$. Note furthermore that if $|Y| = 1$, then the right-hand side of \eqref{e new angle after projection} simplifies to $(\inprod{x_1}{x_2}-\gamma^2)/(1-\gamma^2)$ (this is most easily seen by looking at the second-to-last term in the final equation of the above proof) and that, for fixed $\inprod{x_1}{x_2}$, the latter is a decreasing function in $\gamma^2$.
\end{rem}

In particular, after projecting onto a positive clique (i.e. $\gamma = \alpha$) of size $t$, an angle of $\alpha$ becomes $1/(t+ \alpha^{-1})$ (that is, if $\inprod{x}{x'} = \alpha$, then $\inprod{p_Y(x)}{p_Y(x')} = 1/(t + \alpha^{-1})$) and an angle of $-\alpha$ becomes
\begin{equation*}
  -\frac{2\alpha}{1-\alpha} + \frac{1+\alpha}{(t + \alpha^{-1})(1-\alpha)}.
\end{equation*}
Since these two angles will frequently pop up, we will make the following definition.

\begin{definition}
  For $\alpha \in (0,1)$ and $t \in \mathbb{N}$, let $L(\alpha,t) = \{ - \dleta(1 - \epsilon) + \epsilon,\epsilon\}$, where $\epsilon = \epsilon(\alpha,t)= 1/(t + \alpha^{-1})$ and $\dleta = \dleta(\alpha) = 2\alpha/(1-\alpha)$.
\end{definition}

Note that $L(\alpha,t)$ comprises the two possible angles after projecting onto a positive clique of size $t$. A set attached to a positive clique in a $\pmalpha$-code entirely via positive edges therefore turns into an $L(\alpha,t)$-code after projecting. When we project, we will continue to call edges positive or negative according to whether their original values are $\alpha$ or $-\alpha$. Note in particular that a positive edge may obtain a negative value after projection.

Equipped with this machinery to handle projections, the next lemma gives an upper bound on the number of vertices which are not attached to the positive clique entirely via positive edges. The result is analogous to Lemma 5 of Bukh \cite{B15}.

\begin{lemma}
  \label{l handle the garbage}
  Let $X \cup Y \cup \{z\}$ be a $\pmalpha$-code in $\mathbb{R}^{\dim}$ in which all edges incident to any $y \in Y$ are positive and all edges between $X$ and $z$ are negative. If $|Y|\ge 2/\alpha^2$, then $|X| < 2/\alpha^2$.
\end{lemma}

\begin{proof}
  Let us first project $X \cup \{z\}$ onto the orthogonal complement of $\Span(Y)$, and let us denote $p_{Y}(X)$ by $X'$ and $p_{Y}(z)$ by $z'$. By \Cref{l angle after projecting} and the subsequent paragraph, we verify that $X' \cup \{z'\}$ is an $L(\alpha,|Y|)$-code in which all edges incident to $z'$ are negative and have value $ - \dleta(1-\epsilon) + \epsilon$, which, by \Cref{r angles go down after projecting}, is at most $-\alpha$. The positive angles equal $\epsilon < 1/|Y|$ and since $|Y| \ge 2/\alpha^2$ we get the bound $\epsilon \le \alpha^2/2$. Let us now project $X'$ onto the orthogonal complement of $\Span(z')$. By \Cref{l angle after projecting} and \Cref{r angles go down after projecting} we find that the positive angle becomes
  \begin{equation}
  	\label{e exact angle} 
    \frac{\epsilon - (\dleta(1-\epsilon)-\epsilon)^2}{1 - (\dleta(1-\epsilon)-\epsilon)^2}\textcolor{blue}{\sout{.}}
  \end{equation}
and the negative angles become at most $\frac{-\alpha - (\dleta(1-\epsilon)-\epsilon)^2}{1 - (\dleta(1-\epsilon)-\epsilon)^2}$, which is at most \eqref{e exact angle}.
Furthermore, using $(\dleta ( 1- \epsilon)-\epsilon)^2 \ge \alpha^2$ and \Cref{r angles go down after projecting}, \eqref{e exact angle} is at most $(\epsilon - \alpha^2)/(1 - \alpha^2)$. Since $\epsilon \le \alpha^2/2$, this yields an upper bound of $-\alpha^2/(2-2\alpha^2)$ on all angles after projection. Therefore, after projecting $X'$ onto the orthogonal complement of $\Span(z')$, we obtain a spherical $[-1,-\alpha^2/(2-2\alpha^2)]$-code. By \Cref{l bounding size of a negative clique}, it has size at most $(2-2\alpha^2)/\alpha^2+1 < 2/\alpha^2$, concluding the proof of the lemma.
\end{proof}

Using this lemma, we will see that, after appropriately negating some vertices, all but a fixed number of vertices are attached to the positive clique via positive edges. Hence \Cref{t equiangular} can be reduced to studying $L(\alpha,t)$-codes, as follows.

\begin{lemma} \label{l reduction}
Let $\alpha \in (0,1)$ be fixed and let $t = \log{\log{n}}$. For all sufficiently large $n$ and for any spherical $\pmalpha$-code $C$ in $\R^n$, there exists a spherical $L(\alpha,t)$-code $C'$ in $\R^n$ such that $|C| \leq |C'| + o(n).$
\end{lemma}
\begin{proof} 
Recall that $G_C$ denotes the complete edge-labelled graph corresponding to $C$. From \Cref{l bounding size of a negative clique} we know that $G_C$ doesn't contain a negative clique of size $\alpha^{-1} + 2$. By Ramsey's theorem there exists some integer $R$ such that every graph on at least $R$ vertices contains either a negative clique of size at least $\alpha^{-1} + 2$ or a positive clique of size $t$. A well-known bound of Erd\H{o}s and Szekeres \cite{ES} shows $R \le 4^t = o(n)$. Thus if  $|C| < R$, then we are done by taking $C' = \emptyset$. Otherwise we have by Ramsey's theorem that $G_C$ contains a positive clique $Y$ of size $t$.
   
   For any $T \subset Y$, let $S_T$ comprise all vertices $v$ in $G_C \setminus Y$ for which the edge $vy$ ($y\in Y$) is positive precisely when $y \in T$. Let us negate all vertices $v$ which lie in $S_T$ for some $|T| < t/2$ and note that $C$ remains a $\pmalpha$-code. However, all sets $S_T$ for $|T| < t/2$ are now empty. Given some $T \subset Y$ with $t/2\le |T| <t$, pick a vertex $z \in Y\setminus T$ and consider the $\pmalpha$-code $S_T \cup T \cup \{z\}$. Since any edge incident to $T$ is positive, all edges between $S_T$ and $z$ are negative and $|T| \ge t/2 > 2/\alpha^2$ for $n$ large enough, we can apply \Cref{l handle the garbage} to deduce that $|S_T| < 2/\alpha^2$. Moreover, by \Cref{l angle after projecting} and \Cref{r angles go down after projecting} we have that $C' = p_Y(S_Y)$ is an $L(\alpha, t)$-code with $|C'| = |S_Y|$. Thus we conclude
  \begin{equation*}
    |C| = |S_Y| + \sum_{t/2< |T| < t}{|S_T|} +  |Y| < |C'| +  2^{t+1}/\alpha^2 + t = |C'| + o(n). \qedhere
  \end{equation*}
\end{proof}

\subsection{Spectral techniques}

In view of \Cref{l reduction}, we just need to bound the size of $L(\alpha,t)$-codes. Our main tool will be an inequality bounding the rank of a matrix in terms of its trace and the trace of its square. This inequality goes back to \cite[p. 138]{B97} and its proof is based on a trick employed by Schnirelman in his work on Goldbach's conjecture \cite{S39}. For various combinatorial applications of this inequality, see, for instance, the survey by Alon \cite{A08} and other recent results \cite{DSW14}.

\begin{lemma}
\label{l schnirelman trick}
  Let $M$ be a symmetric real matrix. Then $ \rank{M} \ge \trace{M}^2/\trace{M^2}$.
\end{lemma}

\begin{proof}
  Let $r$ denote the rank of $M$. Since $M$ is a symmetric real matrix, $M$ has precisely $r$ non-zero real eigenvalues $\lambda_1,\dots,\lambda_{r}$. Note that $\trace{M} = \sum_{i=1}^r\lambda_i$ and $\trace{M^2} = \sum_{i=1}^r \lambda_i^2$. Applying Cauchy--Schwarz yields $r\sum_{i=1}^r \lambda_i^2 \ge (\sum_{i=1}^r \lambda_i)^2$, which is equivalent to the desired inequality.
\end{proof}

We use \Cref{l schnirelman trick} to deduce the next claim.

\begin{lemma}
\label{l schnirelman trick applied}
  Let $C$ be an $L(\alpha,t)$-code in $\mathbb{R}^{\dim}$ and let $\average$ denote the average degree of the graph spanned by the negative edges in $G_C$. Then $|C| \le (1 + \dleta^2\average)(\dim+1)$.
\end{lemma}

\begin{proof}
  Recall that $L(\alpha,t) = \{- \dleta(1 - \epsilon)+\epsilon,\epsilon\}$. Every diagonal entry of $N = M_C - \epsilon J$ equals $1- \epsilon$ and $N$ contains exactly $\average |C|$ non-zero off-diagonal entries, each of which equals $-\dleta(1-\epsilon)$. Observe that $\rank{N} \le \rank{M_C} + \rank{J} \le \dim+1$ by the subadditivity of the rank. Furthermore, $\trace{N} = |C|(1-\epsilon)$ and $\trace{N^2} = \sum_{i,j} N_{ij}^2$. By applying \Cref{l schnirelman trick} to $N$ we can therefore deduce
  \begin{equation*}
    |C|^2(1-\epsilon)^2 \le \Bigl(|C|(1-\epsilon)^2 + |C|\average\dleta^2(1-\epsilon)^2\Bigr)(\dim+1),
  \end{equation*}
  which is equivalent to the desired inequality after dividing by $|C|(1-\epsilon)^2$.
\end{proof}

It thus proves necessary to obtain upper bounds on the average degree $d$ of the negative edges in $G_C$.

\begin{rem}
The proof of \Cref{l handle the garbage} provides us with a bound on $d$, and if we are a bit more careful, we already have enough to prove that for a fixed $\alpha$ and $t \rightarrow \infty$, any $L(\alpha,t)$-code has size at most $2n + o(n)$. Indeed, suppose that $C$ is an $L(\alpha,t)$-code with $t \to \infty$. Let $z' \in C$ and let $X'$ be the vertices connected to $z'$ via negative edges. We project $X'$ onto the orthogonal complement of $z'$ and observe that since $t \rightarrow \infty$, $\epsilon \rightarrow 0$ and hence the positive angle \eqref{e exact angle} in \Cref{l handle the garbage} becomes $-\sigma^2 / (1 - \sigma^2) + o(1)$. Thus we obtain a $[-1, -\sigma^2 / (1 - \sigma^2) + o(1)]$-code which has size at most $(1 - \sigma^2)/\sigma^2 + 1 + o(1) = 1 / \sigma^2 + o(1)$ by \Cref{l bounding size of a negative clique}. Since this holds for all $z$, we have that $d \leq 1 / \sigma^2 + o(1)$ and hence applying \Cref{l schnirelman trick applied} we conclude $|C| \leq 2n + o(n)$.
\end{rem}

The following lemma shows that it will be sufficient to find an upper bound on $\average$ in terms of the largest eigenvalue of some fixed-size subgraph of $C$, by which we mean a subgraph of size $O(1)$. Let us fix some standard notation. For a matrix $A$, we denote its largest eigenvalue by $\lambda_1(A)$. If $H$ is a graph, then we can identify $H$ with its adjacency matrix $A(H)$, so that we will write $\lambda_1(H)$ to mean $\lambda_1(A(H))$. It is well-known that $\lambda_1$ is monotone in the following sense: if $H$ is a subgraph of $G$, then $\lambda_1(H) \le \lambda_1(G)$ (see e.g. \cite[chapter 11, exercise 13]{L14})

\begin{lemma}
\label{l relating angle and eigenvalue}
  Let $C$ be a fixed-size $L(\alpha,t)$-code in $\mathbb{R}^{\dim}$ and assume that $t \to \infty$ as $n \to \infty$. Let $H$ be the subgraph of $G_C$ containing precisely all negative edges. Then $\dleta \lambda_1(H) \le 1 + o(1)$.
\end{lemma}

\begin{proof}
  The Gram matrix $M_C$ and $A = A(H)$ are related by the equation
  \begin{equation*}
    M_C = I + \epsilon (J-I) - \dleta(1 - \epsilon)A,
  \end{equation*}
  where $J$ denotes the all-ones matrix. Let $x$ be a normalised eigenvector of $A$ with eigenvalue $\lambda_1(H)$. Since $M_C$ is positive-semidefinite, we deduce
  \begin{equation}
  \label{e fundamental eigenvalue inequality}
    0 \le \inprod{M_Cx}{x} = 1 - \epsilon + \epsilon \inprod{Jx}{x} - \dleta(1-\epsilon) \lambda_1(H) \le 1 - \dleta \lambda_1(H) + \epsilon(|C| +\dleta \lambda_1(H)),
  \end{equation}
  where $\inprod{Jx}{x}\le |C|$ follows from the fact that $|C|$ is the largest eigenvalue of $J$. Since $\dleta, |C|$ and $\lambda_1(H)$ are all $O(1)$ and $\epsilon = o(1)$, \eqref{e fundamental eigenvalue inequality} yields the required $\dleta \lambda_1(H) \le 1 + o(1)$.
\end{proof}

The following two lemmas are concerned with establishing a connection between the average degree of a graph and its largest eigenvalue. The first lemma and its proof are inspired by Nilli's proof \cite{N91} of the Alon-Boppana bound on the second eigenvalue of a graph.

\begin{lemma}
\label{l maximal eigenvalue of depth k tree}
  Let $G$ be a graph with minimum degree $\delta>1$. Let $v_0$ be some vertex of $G$ and let $H$ be the subgraph consisting of all vertices within distance $\wellsee$ of $v_0$. Then $\lambda_1(H) \ge 2(1 - 1/(\wellseeplusone))\sqrt{\delta - 1}$.
\end{lemma}

\begin{proof}
  For $0\le i \le \wellsee$, let $V_i$ denote the set of vertices at distance $i$ from $v_0$ in $H$, let $e_i$ denote the number of edges in $H[V_i]$ and let $\edges_i$ denote the number of edges in $H[V_i, V_{i+1}]$, where we set $\edges_{\wellsee} = 0$ and, since we will need it later in the proof, $\edges_{-1} = 0$. Let us define a function $\func$ on the vertices of $H$ by $\func(v) = (\delta-1)^{-i/2}$ if $v\in V_i$. Letting $A$ denote the adjacency matrix of $H$, we have $\lambda_1(H) \ge \inprod{A\func}{\func}/\inprod{\func}{\func}$. In order to prove the desired bound on $\lambda_1(H)$, we therefore need to bound the quantity $\inprod{\func}{\func}$ from above in terms of $\inprod{A\func}{\func}$. We have
  \begin{equation*}
    \inprod{\func}{\func} = \sum_{i=0}^{\wellsee} \frac{|V_i|}{(\delta-1)^i} \quad
    \text{and} \quad
    \inprod{A\func}{\func} = \sum_{i=0}^{\wellsee}\biggl( \frac{2e_i}{(\delta-1)^i} + \frac{2\edges_i}{(\delta-1)^{i+1/2}}\biggr).
  \end{equation*}
  Note that for $0\le i\le \wellseeminusone$, $\edges_{i-1} + 2e_i + \edges_i$ counts the sum of the degrees of all vertices in $V_i$ and is therefore of size at least $\delta |V_i|$. Moreover, since every vertex in $V_{i+1}$ is adjacent to some vertex in $V_i$ we have $|V_{i+1}| \le \edges_i$. Fix any $j$ in the range $0\le j \le \wellsee$. Using the above two observations we find
  \begin{equation}
  \label{e upper bound on (f,f)}
    \inprod{\func}{\func} \le \sum_{i=0}^{j-1} \frac{\edges_{i-1} + 2e_i + \edges_i}{\delta(\delta-1)^i} + \frac{|V_j|}{(\delta - 1)^{j}} + \sum_{i = j}^{\wellseeminusone} \frac{\edges_i}{(\delta-1)^{i+1}}.
  \end{equation}
  Observe that $\delta \ge 2\sqrt{\delta - 1}$ and that we have the identity
  \begin{equation*}
    \frac{1}{\delta(\delta-1)^i} + \frac{1}{\delta(\delta-1)^{i+1}} = \frac{1}{(\delta-1)^{i+1}}.
  \end{equation*}
  Collecting terms belonging to the same $\edges_i$ in the first sum of \eqref{e upper bound on (f,f)} and using the estimate and identity of the previous sentence, we find
  \begin{equation*}
    \inprod{\func}{\func} \le \sum_{i=0}^{\wellsee} \biggl(\frac{\edges_i}{(\delta-1)^{i+1}} + \frac{2e_i}{\delta(\delta-1)^i} \biggr) + \frac{|V_j|}{(\delta - 1)^{j}} \le \frac{\inprod{A\func}{\func}}{2\sqrt{\delta - 1}} + \frac{|V_j|}{(\delta - 1)^j}.
  \end{equation*}
  Averaging over all $0 \le j \le \wellsee$ yields
  \begin{equation*}
    \inprod{\func}{\func} \le \frac{\inprod{A\func}{\func}}{2\sqrt{\delta - 1}} + \frac{\inprod{\func}{\func}}{k+1},
  \end{equation*}
  which is equivalent to the desired inequality.  
\end{proof}

\begin{lemma}
  \label{l eigenvalues of graphs}
  Let $H$ be a connected graph on
  \begin{enumerate}[label=(\roman*)]
    \item $\numberofparts$ vertices and $\numberofpartsminusone$ edges. Then $\lambda_1(H) \ge \twotimesnumberofparts/\numberofparts$.
    \item $k$ vertices and $k$ edges. Then $\lambda_1(H) \ge 2$.
    \item $6$ vertices and $5$ edges, so that some vertex has degree $5$. Then $\lambda_1(H) \ge 2.2$.
    \item $5$ vertices and $5$ edges so that some vertex has degree $4$. Then $\lambda_1(H) \ge 2.25$.
    \item $8$ edges so that some vertex has degree $4$.
      Then $\lambda_1(H) \ge 2.2$.
  \end{enumerate}
\end{lemma}

\begin{proof}
  Let $A$ denote the adjacency matrix of $H$ and $\allonesvector$ the all-ones vector of appropriate length. Note that $\lambda_1(H) \ge \inprod{A\allonesvector}{\allonesvector} /\inprod{\allonesvector}{\allonesvector} = d$, where $d$ denotes the average degree of $H$. This is sufficient to establish (i) and (ii), since the average degree of the graphs is $\twotimesnumberofparts/\numberofparts$ and $2$, respectively.
  
  Suppose that $H$ is a star with $5$ leaves, as in (iii). Let $x$ be the vector giving weight $\sqrt{5}$ to its internal vertex and weight $1$ to each leaf. Then $\inprod{x}{x} = 10$ and $\inprod{Ax}{x} = 10\sqrt{5}$ yielding the required $\lambda_1(H) \ge \sqrt{5} > 2.2$.
  
  Suppose that $H$ is as in (iv). Let $x$ be the vector giving weight $1$ to the vertex of degree $4$ and $1/2$ to the others. Then $\inprod{x}{x} = 2$ and $ \inprod{Ax}{x} = 4.5$ yielding the required $\lambda_1(H) \ge 2.25$.
  
  Finally, suppose that $H$ is as in (v) and let $v$ be the vertex of degree $4$. If two of the neighbours of $v$ are adjacent, we are done by (iv). Otherwise, let $x$ be the vector giving weight $4$ to $v$, weight $\sqrt{5}$ to its $4$ neighbours and weight $1$ to all other vertices. Then $\inprod{Ax}{x} = 40\sqrt{5}$ and $\inprod{x}{x} \le 40$ since there are at most $4$ vertices of weight $1$. Hence $\lambda_1(H) \ge \inprod{Ax}{x}/\inprod{x}{x} \ge \sqrt{5} > 2.2$.
\end{proof}

The next lemma deals with $\pmalpha$-codes in which the negative edges are very sparse. This will be the case when $\alpha$ is rather large.

\begin{lemma}
  \label{l full rank}
  Let $\alpha \in (0,1)\setminus \{ \frac{1}{3}\}$ and let $C$ be an $L(\alpha,t)$-code in $\mathbb{R}^{\dim}$. If the negative edges form a matching, then $|C| \le \dim+1$.
\end{lemma}

\begin{proof}
  Recall that $L(\alpha, t) = \{-\dleta (1-\epsilon) + \epsilon,\epsilon\}$. Let $J$ denote the all-ones matrix. Since the rank of matrices is subadditive, we have
  \begin{equation}
  \label{e rank of matching}
    \rank{M_C-\epsilon J} \le \rank{M_C} + \rank{-\epsilon J} = \rank{M_C} + 1 \le \dim + 1.
  \end{equation}
  Since the negative edges of $G_C$ form a matching, the matrix $(M_C - \epsilon J)/(1-\epsilon)$ consists of $m$ identical $2\times 2$ blocks with $1$'s on the diagonal and $-\dleta$ off the diagonal, and $|C|-2m$ identical $1 \times 1$ identity matrices, where $m$ denotes the number of negative edges. The former have determinant $1 - \dleta^2$, the latter $1$. Since $\alpha \ne \frac{1}{3}$, these quantities are non-zero, so that $M_C - \epsilon J$ has full rank, that is, $\rank{M_C - \alpha J} = |C|$. Together with \eqref{e rank of matching} this gives the desired inequality. 
\end{proof}

\begin{rem}
  Note that one can also prove $|C| \le \dim$ with some more work.
\end{rem}

\subsection{Proof of the main result}

In this section, we present the proof of \Cref{t equiangular}. First, combining \Cref{l schnirelman trick applied} with the newly gained information about the relation between $\dleta$, the largest eigenvalue of fixed-size graphs and $\average$, we prove the following theorem about $L(\alpha, t)$-codes. This theorem will allow us to analyse equiangular lines for all angles except $\arccos \frac{1}{3}$.

\begin{theorem}
\label{l projected bound}
  Let $\alpha \in (0,1)\setminus\{\frac{1}{3}\}$ and $t \in \mathbb{N}$ so that $t\to \infty$ as $n \to \infty$. Let $C$ be an $L(\alpha,t)$-code in $\mathbb{R}^{\dim}$ for which every vertex is incident to at most $O(1)$ negative edges. Then $|C| < \finalconstant \dim$ for sufficiently large $\dim$.
\end{theorem}

\begin{proof}
  Recall that $L(\alpha,t) = \{- \dleta(1 - \epsilon)+\epsilon,\epsilon\}$, where $\epsilon= 1/(t + \alpha^{-1})$ and $\dleta = 2\alpha/(1-\alpha)$; note that $\epsilon = o(1)$. Throughout the proof, let $G$ denote the graph consisting only of the negative edges of the graph corresponding to $C$ (that is, we delete from $G_C$ all positive edges to obtain $G$). We split the proof of this lemma into different regimes, depending on the value of $\dleta$.
  
  \textbf{Case 1, $\dleta \in [0.71, \infty)$:} We will show that no two edges in $G$ are adjacent. Together with \Cref{l full rank} this will show that $|C| \le \dim + 1$. Let $\betta = -\dleta ( 1- \epsilon) + \epsilon$. If $\betta < - 1$, $G$ cannot contain any edges and we are done. Otherwise, suppose to the contrary that $x, y$ and $z$ are unit vectors in $C$ so that $xy$ and $xz$ are negative edges. Let us decompose $y$ and $z$ as $y = \betta x + u$ and $z = \betta x + v$, where $u$ and $v$ are orthogonal to $x$. Since $x, y$ and $z$ are unit vectors, taking norms on both sides of each equation and rearranging yields $1 - \betta^2 =\norm{u}^2 = \norm{v}^2$. Since $\dleta \ge 0.71 > 1/\sqrt{2}$ and $\epsilon = o(1)$, we have $\betta^2 > 1/2 + \epsilon$ for sufficiently large $\dim$ and hence $t$. Therefore, $\norm{u}, \norm{v} < 1/\sqrt{2}$. Furthermore, taking the inner product of $y$ and $z$ gives
  \begin{equation*}
    \inprod{y}{z} = \betta^2 + \inprod{u}{v} > \epsilon + 1/2 - \norm{u}\norm{v} > \epsilon,
  \end{equation*}
  a contradiction to $\inprod{y}{z} \in L(\alpha,t)$, finishing the proof of the first case.
  
  \textbf{Case 2, $\dleta \in [0.551, 0.71]$:} We will prove that $G$ decomposes into trees on at most $\numberofpartsminusone$ vertices. \Cref{l relating angle and eigenvalue} shows that $G$ cannot contain a fixed-size subgraph $H$ with $\lambda_1(H) > 1/0.55 = 20/11$. In particular, by \Cref{l eigenvalues of graphs}, $G$ doesn't contain a subgraph on $\numberofparts$ vertices and $\numberofpartsminusone$ edges or a subgraph on $k$ vertices and $k$ edges for any $k \le \numberofpartsminusone$. Since any connected graph on at least $11$ vertices contains a tree on $11$ vertices, all components have at most $10$ vertices, and since the only acyclic components are trees, all components are trees on at most $\numberofpartsminusone$ vertices. The average degree of any component is therefore at most $\twotimesnumberofpartsminusone/\numberofpartsminusone$ and hence so is the average degree of $G$. Applying \Cref{l schnirelman trick applied} establishes the required bound
  \begin{equation*}
    |C| \le (1 + 1.8 \dleta^2)(\dim + 1) < \finalconstant \dim.
  \end{equation*}
  
  \textbf{Case 3, $\dleta \in [0.47, 0.551]$:} \Cref{l relating angle and eigenvalue} implies that $G$ cannot contain a fixed-size subgraph $H$ with $\lambda_1(H) > 2.13 > 1/0.47$. We can therefore deduce from \Cref{l eigenvalues of graphs} that $G$ doesn't contain a vertex of degree higher than $4$, that the neighbourhood of a vertex of degree $4$ contains no edges and that the neighbourhood of a vertex of degree $4$ is incident to at most $3$ more edges. The latter two properties imply that each vertex of degree $4$ is adjacent to a leaf. On the other hand, each leaf is adjacent to exactly one vertex (not necessarily of degree $4$), so $G$ contains no more vertices of degree $4$ than leaves. Since $G$ also doesn't contain any vertices of higher degree than $4$, the average degree of $G$ is at most $3$. Applying \Cref{l schnirelman trick applied} establishes the required bound
  \begin{equation*}
    |C| \le (1+ 3\dleta^2) (\dim+1) < \finalconstant \dim.
  \end{equation*}
  
  \textbf{Case 4, $\dleta \in (0,0.47]$:} Let $d$ be the average degree of the negative edges in $G_C$ and suppose for the sake of contradiction that $|C| > \finalconstant \dim$. Combining this lower bound on $|C|$ with the upper bound given by \Cref{l schnirelman trick applied} yields $\average > \finalconstantminusone/\dleta^2 - o(1) > 4$. Let $l$ be the integer satisfying $2l< \average \le 2l+2$; note that $\average > 4$ implies $l\ge 2$. It is well known that a graph with average degree $\average$ contains a subgraph with minimum degree at least $\average/2$. Hence $G$ contains a subgraph $G'$ with minimum degree at least $l+1$. Applying \Cref{l maximal eigenvalue of depth k tree} to $G'$ for $k = 11$, we find that $G'$ contains a subgraph $H$ with maximal eigenvalue $\lambda_1(H) > 1.83\sqrt{l}$ and, since the maximum degree of $G$ is bounded by a constant independent of $\dim$, so is the size of $H$ by construction. \Cref{l relating angle and eigenvalue} then gives
  \begin{equation*}
    \dleta^2 \le \frac{1 + o(1)}{1.83^2l} < \frac{1}{3.34l},
  \end{equation*}
  which together with \Cref{l schnirelman trick applied} and $l\ge 2$ yields the required
  \begin{equation*}
    |C| < \biggl( 1 + \frac{2(l+1)}{3.34l} \biggr)(\dim +1 ) < \finalconstant \dim. \qedhere
  \end{equation*}
\end{proof}

Now that we have finished all the necessary preparation, we are ready to complete the proof of our first theorem.

\begin{proof}[Proof of \Cref{t equiangular}]
 Let $C$ be a $\pmalpha$-code in $\R^n$ and let $t = \log\log \dim$. Suppose first that $\alpha \ne \frac{1}{3}$. Then by \Cref{l reduction}, there exists an $L(\alpha,t)$-code $C'$ in $\R^n$ such that $|C| \leq |C'| + o(n)$. By \Cref{l projected bound}, we have $|C'| < 1.92n$ and hence $|C| \leq 1.93n$ for $n$ large enough.

  Otherwise $\alpha = \frac{1}{3}$. For a detailed proof of the upper bound of $2\dim -2$, we refer the reader to \cite{LS73}. Let us nonetheless sketch it for the sake of completeness. Note that what follows is only an outline; filling in all the details requires substantially more work. Instead of finding a large positive clique, we consider the largest negative clique $M$ in any graph obtained from $G_C$ by switching any number of vertices. By \Cref{l bounding size of a negative clique}, we have that $|M| \leq 4$. We can then show that the cases $|M| \le 3$ are either straightforward or can be reduced to the case $|M|=4$. In the latter case, we can show that unless all vertices attach to $M$ in the same way (that is, no two vertices outside the clique attach to some vertex within the clique differently), $|C|$ is bounded from above by some constant independent of $\dim$. If they do all attach in the same way, then if we consider the projection $C' = p_M(C \backslash M)$ of $C \backslash M$ onto the orthogonal complement of $\Span(M)$, we obtain a $\{-1,0\}$-code. This means that any two distinct vectors of $C'$ are either orthogonal or lie in the same $1$-dimensional subspace, so that $|C'| \leq 2\text{dim}(C')$. Moreover, by \Cref{r simplex} $M$ is a regular simplex so it lives in a $3$-dimensional subspace, and hence $\text{dim}(C') = n-3$. Thus $|C| = |M| + |C'| \leq 4 + 2(n-3) = 2n -2$, finishing the proof.
\end{proof}

\section{Spherical codes}
\label{s spherical codes}

Let us now turn our attention from equiangular sets of lines to the more general setting of spherical codes. Recall that a spherical $\spherical$-code is a finite non-empty set $\code$ of unit vectors in Euclidean space $\mathbb{R}^\dim$ so that $\inprod{x}{y} \in \spherical$ for any pair of distinct vectors $x,y$ in $\code$. In this section, we prove \Cref{t spherical} in the case $k = 1$, obtaining the asymptotically tight bound even when $\alpha$ is allowed to depend on $n$. The proof features all ideas central to the argument in the multi-angular case (which we will treat in detail in \Cref{s multi-angular spherical codes}), without concealing them unnecessarily. 

\begin{theorem}
  \label{t spherical 1}
  Let $\beta \in (0,1]$ be fixed and $\alpha \in [-1, 1)$. Then any $[-1,-\beta]\cup \{ \alpha\}$-code in $\mathbb{R}^\dim$ has size at most
  \begin{equation*}
    2 \Bigl ( 1 + \max\Bigl( \frac{\alpha}{\beta}, 0 \Bigr)\Bigr) \dim + o(\dim).
  \end{equation*}
\end{theorem}

Since an equiangular set of lines corresponds to a $\pmalpha$-code, this implies a weaker bound of $4\dim + o(\dim)$ for equiangular sets. The reason for this is that we can't switch edges from negative to positive any more, since a negative edge might not obtain value $\alpha$ after switching. Moreover, this is essentially tight because if we take our example of $2n-2$ lines with angle $\arccos \frac{1}{3}$ and take both unit vectors along each line, we get a $[-1, -\frac{1}{3}] \cup \{ \frac{1}{3} \}$-code of size $4n - 4$.

The beginning of the proof of \Cref{t spherical 1} is along the lines of the proof of the corresponding theorem for equiangular sets of lines. We start by finding a large positive clique in $G_C$. Unlike before, however, a substantial portion of the vertices might not attach to this clique entirely via positive edges. In fact, almost all vertices attach either entirely via positive edges or mostly via negative ones. Similarly to before, we can bound the size of the set of vertices attaching positively to the clique by $2\dim + o(\dim)$. Repeating this argument yields a set of positive cliques in such a way that almost all edges between these cliques are negative. This imposes a bound on the number of repetitions, which is enough to bound the size of the $\spherical$-code.

We start by proving a lemma similar to \Cref{l handle the garbage}, which enables us to analyse how vertices connect to a positive clique.

\begin{lemma}
    \label{l one angle garbage}
  Let $L = [-1, -\beta] \cup \{ \alpha\}$ for some $\alpha, \beta \in (0,1)$ and suppose that $X \cup Y \cup \{z\}$ is an $L$-code in which all edges incident to any $y \in Y$ are positive edges and all edges between $X$ and $z$ are negative edges. Suppose furthermore that $|Y| > 1/\alpha^2$. Then $|X| < 1/\beta^2$.
\end{lemma}

\begin{proof}
 Let $\alpha_X$ denote the average value of the edges in $X$ and $-\beta_z$ the average value of the edges between $X$ and $z$. Note that $\alpha_X \le \alpha$ and $\beta_z \geq \beta$. Let $M$ be the Gram matrix of $X \cup Y \cup \{z \}$ and let $ v = (x, \dots, x, y, \dots, y, \zeta)^\intercal$, where
  \begin{equation*}
    x = 1/|X|, \quad y = -\frac{\alpha(1+\beta_z)/|Y|}{\alpha - \alpha^2 + (1 -\alpha)/|Y|} \quad \text{and} \quad \zeta = \beta_z - y |Y|\alpha.
  \end{equation*}
  Then
  \begin{equation*}
  \begin{split}
     \inprod{Mv}{v} &= \frac{(|X|^2-|X|)\alpha_X+|X|}{|X|^2} + 2y|Y|\alpha + y^2((|Y|^2-|Y|)\alpha + |Y|) + 2\zeta(-\beta_z + y|Y|\alpha) + \zeta^2 \\
     &= \frac{1-\alpha_X}{|X|}+\alpha_X -\beta_z^2 + 2y|Y|\alpha(1 + \beta_z) + y^2|Y|^2\bigl(\alpha -\alpha^2 + (1-\alpha)/|Y|\bigr) \\
     &= \frac{1-\alpha_X}{|X|}+\alpha_X -\beta_z^2 - \frac{\alpha^2(1 + \beta_z)^2}{\alpha -\alpha^2 + (1-\alpha)/|Y|} \\
     &\le \frac{1-\alpha}{|X|}+\alpha -\beta_z^2 - \frac{\alpha^2(1 + \beta_z)^2}{\alpha -\alpha^2 + (1-\alpha)/|Y|}.
  \end{split}
  \end{equation*}
  Since $\inprod{M v}{v} \ge 0$ and $\beta_z^2 \ge \beta^2$, it is therefore sufficient to prove that
  \begin{equation*}
    \alpha -\beta_z^2 - \frac{\alpha^2(1 + \beta_z)^2}{\alpha -\alpha^2 + (1-\alpha)/|Y|} < -\beta_z^2(1-\alpha).
  \end{equation*}
  Using $|Y| > 1/\alpha^2$ and rewriting the above inequality, it suffices to show that
  \begin{equation*}
    \alpha(1-\beta_z^2) < \frac{\alpha(1+\beta_z)^2}{1-\alpha^2},
  \end{equation*}
  which is clearly true since $\alpha, \beta_z > 0$.
\end{proof}

\begin{rem}
The $v$ in the above proof is chosen so as to minimise $\inprod{M v}{v} / ||v||^2$. An appropriate projection also minimises this quantity and so the above argument could also be done using projections. Indeed, this minimisation is precisely why projections are so useful for us.
\end{rem}

After projecting onto a large $\alpha$-clique, the new $\alpha$ will become $o(1)$. In this case, the next lemma gives a bound on the values of the negative edges incident to a fixed vertex.

\begin{lemma}
  \label{l beta edges}
  Let $\spherical = [-1, - \beta] \cup \{\alpha\}$ and let $C$ be an $\spherical$-code. If $\alpha = o(1)$ and $-\beta_1, \dots, -\beta_{N}$ are the values of the negative edges incident to some vertex $x$ in $G_C$, then $\alpha N = o(1)$ and $\sum_{i = 1}^{N} \beta_i^2 \le 1 + o(1)$.
  \end{lemma}

\begin{proof}
We will first derive the upper bound on $N$.  Let $C = N_\beta(x) \cup \{x\}$ and $M = M_C$. If we let $v = (1, \dots, 1, \beta N)^\intercal$, then we have
  \begin{equation*}
    0 \le \inprod{M v}{v} = \sum_{1\le i,j \le N} M_{ij} - 2 \beta N \sum_{i=1}^N \beta_i+\beta^2 N^2 \le N + o(1)N^2 - \beta^2N^2,
  \end{equation*}
  which implies $N \leq (1 + o(1))\beta^{-2}$ and therefore establishes the claimed $\alpha N \le (1+o(1))\alpha\beta^{-2} = o(1)$. Now if we let $w = (\beta_1, \dots, \beta_N, 1)^\intercal$, then we obtain
  \begin{equation*}
    \inprod{Mw}{w}
    = 1 - \sum_{i=1}^N \beta_i^2 + \sum_{\substack{1\le i,j \le N \\ i \ne j}} \beta_i\beta_j M_{ij}
    \le 1 - \sum_{i=1}^N \beta_i^2 + \sum_{1\le i,j \le N} \beta_i\beta_j \alpha
    \le 1 - \sum_{i=1}^N \beta_i^2 + \alpha N \sum_{i=1}^N \beta_i^2,
 \end{equation*}
 where the last step follows from Cauchy--Schwarz. Using $\inprod{Mw}{w} \ge 0$ and $\alpha N = o(1)$, we obtain the required $\sum_{i=1}^N \beta_i^2 \le 1 + o(1)$.
\end{proof}

As we outlined above, when proving \Cref{t spherical 1} we will obtain a multipartite graph which has mostly negative edges between its parts. The next lemma gives a bound on the number of parts of such a graph. Because we will consider more general spherical codes in a later section, we prove it in more generality.

\begin{lemma}
  \label{l multipartite graphs}
 Let $\beta \in (0,1]$ be fixed, let $\alpha \in [-1, 1)$ and $L = [-1, -\beta] \cup [\alpha, 1)$. Suppose $t \rightarrow \infty$ as $n \rightarrow \infty$ and let $C$ be a spherical $L$-code such that $G_C$ is the disjoint union of $\ell$ $\alpha$-cliques $Y_1, \ldots, Y_{\ell}$ each of size $t$, such that the number of $\beta$-edges between any $Y_i$ and $Y_j$ is at least $t^2(1 - o(1))$. Then $\ell \le 1 + \alpha / \beta + o(1)$.
\end{lemma}

\begin{proof}
Let $A$ be the number of $\alpha$-edges and $B$ be the number of $\beta$-edges in $G_C$. Since the remaining $\binom{|C|}{2} - A - B$ edges have value at most $1$, we have as in the proof of \Cref{l bounding size of a negative clique} that
\[ 0 \leq \norm[\Big]{\sum_{x \in C}{x}}^2 \leq |C| - 2B \beta + 2 A \alpha + |C|(|C|-1) - 2B - 2A ,\]
which implies $2 B ( \beta + 1) + 2 A ( 1 - \alpha ) \leq |C|^2$. Now observe that $C$ has size $\ell t$, $\ell \binom{t}{2}$ $\alpha$-edges inside the parts, and at least $ \tbinom{\ell}{2}t^2 (1 - o(1))$ $\beta$-edges between parts. Thus if we substitute these values into the above inequality and solve for $\ell$, we obtain the required
  \begin{equation*}
    \ell \leq \frac{ \beta + \alpha + \frac{1 - \alpha}{t} - o(1)(1 + \beta)}{\beta - o(1)(1 + \beta)} = 1 + \frac{\alpha}{\beta} + o(1). \qedhere
  \end{equation*}
\end{proof}

Now we have all of the necessary tools to prove the main theorem of this section.

\begin{proof}[Proof of \Cref{t spherical 1}]
  Suppose first that $|\alpha| < 1/\log \log n = o(1)$. Let $\polygram = M_C - \alpha J$. By the subadditivity of the rank we have $\rank{\polygram} \leq \rank{M_C} + \rank{J} \le \dim + 1$. Consider some $x \in C$. Let $N = d_{\beta}(x)$ and let $-\beta_1, \ldots, -\beta_N$ be the values of the negative edges incident to $x$. By \Cref{l beta edges} we have $\sum_{j=1}^N \beta_j^2 \le 1 + o(1)$ and $\alpha N = o(1)$. It follows that if $i$ is the row corresponding to $x$ in $N$, then
  \begin{equation*}
    \sum_{j \neq i}{\polygram_{i,j}^2} 
    = \sum_{j=1}^N{(\beta_j + \alpha)^2}
    \le \sum_{j=1}^N{\beta_j^2} + 2 |\alpha| N + \alpha^2 N
    \le 1 + o(1).
  \end{equation*}
Noting that $Q$ has $ 1 - \alpha$ on the diagonal, we obtain
\[ \tr(\polygram^2) = \sum_{i = 1}^{|C|}{Q_{i,i}^2} + \sum_{i = 1}^{|C|}{\sum_{j \neq i}{ Q_{i,j}^2 }} \leq |C| (1 - \alpha) + |C| \left(1 + o(1)) \right) \leq |C|(2 + o(1)). \]
Thus applying \Cref{l schnirelman trick} to $Q$ yields  
  \begin{equation*}
    |C|^2(1-\alpha)^2 = \trace{\polygram}^2 \le \trace{\polygram^2} \rank{\polygram} \le |C|(2 + o(1)) (\dim + 1) .
  \end{equation*}
  After dividing by $|C|(1-\alpha)^2 = |C|(1-o(1))$, we obtain the required $|C| \le 2\dim +o(\dim)$.
  
  We now prove the theorem for all remaining values of $\alpha$, that is, for all $\alpha$ satisfying $|\alpha| \ge  1/ \log \log n$. If $\alpha < 0$ we are done by \Cref{l bounding size of a negative clique}. Suppose therefore that $\alpha >0$. Let $\ell = 1 + \alpha/\beta$ and $t = \tfrac{1}{4}\log \dim$. Suppose for the sake of contradiction that there exists some $\epsilon > 0$ so that for arbitrarily large $\dim$,
  \begin{equation*}
    |C| > 2\ell(1 + 2\epsilon) \dim.
  \end{equation*}
  
  From \Cref{l bounding size of a negative clique} we know that $G_C$ doesn't contain a negative clique bigger than $\beta^{-1} + 2$. By Ramsey's theorem, there exists some integer $R$ so that every graph on at least $R$ vertices contains either a negative clique of size at least $\beta^{-1} + 2$ or a positive clique of size $t$. A well-known bound of Erd\H{o}s and Szekeres \cite{ES} shows $R \le 4^t \le \sqrt{\dim} < |G_C|$. Therefore $G_C$ contains a positive clique $Y$ of size $t$.
  
  For any $T \subset Y$, let $S_T$ comprise all vertices $v$ in $G_C \setminus Y$ for which $vy$ ($y\in Y$) is an $\alpha$-edge precisely when $y \in T$. Given some $T \subset Y$ with $\sqrt{t} \le |T| <t$, pick a vertex $z \in Y\setminus T$ and consider the $[-1,-\beta] \cup \{\alpha\}$-code $S_T \cup T \cup \{z\}$. Since any edge incident to $T$ is an $\alpha$-edge, all edges between $S_T$ and $z$ are $\beta$-edges and $|T| \ge \sqrt{t} > 1/\alpha^2$, we can apply \Cref{l one angle garbage} to deduce that $|S_T| < 1/\beta^2$. For $T=Y$, since $p_Y(S_Y)$ is a $[-1,-\beta] \cup \{\alpha'\}$-code for $\alpha' = 1/(t+1/\alpha) <  1/\log \log n$,  we infer $|S_Y| < (2+ \epsilon )\dim$ for sufficiently large $\dim$ from the first part of the proof. Now let
  \begin{equation*}
    G' = G_C \setminus \biggl( Y \cup \bigcup_{\substack{T \subset Y \\ |T| > \sqrt{t}}}{S_T} \biggr)
  \end{equation*}
  and note that for all $x \in G'$, $|N_{\alpha}(x) \cap Y| = o(t)$. Applying the bounds derived above, we obtain
  \begin{equation*}
    \biggl| Y \cup \bigcup_{\substack{T \subset Y \\ |T| > \sqrt{t}}}{S_T} \biggr| \le t + 2^{t}/\beta^2 + |S_Y| \le 2(1+\epsilon) \dim.
  \end{equation*}
  We can therefore iterate this procedure $\ell$ times to obtain $\ell$ disjoint $\alpha$-cliques $Y_1, \ldots, Y_{\ell}$ and a disjoint graph $G'$ of size at least $2\epsilon n > \sqrt{\dim}$, so that the number of $\alpha$-edges between $Y_i$ and $Y_j$ is $o(t^2)$ for distinct $i$ and $j$. Since $|G'| > \sqrt{\dim}$, there exists an additional $\alpha$-clique $Y_{\ell + 1} \subset G'$ of size $t$, also with $o(t^2)$ edges to any $Y_i$. But then the induced subgraph on $Y_1 \cup \ldots \cup Y_{\ell + 1}$ contradicts \Cref{l multipartite graphs}, finishing the proof.
\end{proof}

\section{A construction}
\label{s construction}

In this section we prove \Cref{t construction}, which states that for any positive integers $n, k, r$ and $ \alpha_1 \in (0,1)$ with $k$ and $\alpha_1$ fixed and $r \leq \sqrt{n}$, there exist $\alpha_2, \ldots, \alpha_k$, $\beta = \alpha_1 / r - O(\sqrt{\log(n) / n})$ and a spherical $L$-code of size $(1 + r) \binom{n}{k}$ in $\R^{n + r}$ with $L = [-1, -\beta] \cup \{ \alpha_1, \ldots, \alpha_k \}$. This construction shows that the second statement of \Cref{t spherical} is tight up to a constant factor. It also answers a question of Bukh. In \cite{B15} he asked whether for fixed $\alpha$, any spherical $[-1, 0) \cup \{\alpha\}$-code has size at most linear in the dimension. \Cref{t construction} gives an example of such a code with size that is superlinear in the dimension. Indeed, for any $\alpha$ fixed, if we choose $r = \sqrt{n} / \log(n)$ then by \Cref{t construction}, we obtain a $[-1, -\beta] \cup \{\alpha\}$-code of size at least $r n = n^{3/2} / \log{n}$ in $\R^{(1+o(1))n}$, where $\beta > 0$ for $n$ large enough.

Given vectors $u \in \R^n$ and $v \in \R^m$, we let $(u,v)$ denote the concatenated vector in $\R^{n + m}$. We first give an outline of the construction. We start by finding a $\{0, 1/k, \ldots, (k-1)/k\}$-code $C$ of size $\binom{n}{k}$, given by \Cref{l k-code}. We then take a regular $r$-simplex so that all inner products are negative. For each vector $v$ of the simplex, we take a randomly rotated copy $C_v$ of $C$ and attach a scaled $C_v$ to $v$ by concatenation, and then normalise all vectors to be unit length. That is, for all $u \in C_v$ we take $(\lambda u,v) / \sqrt{\lambda^2 + 1}$ where $\lambda$ is a scaling factor chosen so that the resulting code has the given $\alpha_1$ as one of its inner products. By randomly rotating the copies of $C$, we ensure that the inner products between vectors coming from different copies remain negative. This follows from the well-known fact that the inner product between random unit vectors is unlikely to be much bigger than $1/\sqrt{n}$, given by \Cref{l random rotation}.

\begin{lemma} \label{l k-code}
For any positive integers $n,k$ with $k \leq n$, there exists a spherical $\{0, 1/k, \ldots, (k-1)/k\}$-code of size $\binom{n}{k}$ in $\mathbb{R}^\dim$.
\end{lemma}
\begin{proof}
Let $C$ be the set of $\{0,1\}$-vectors in $\R^n$ having exactly $k$ 1's. Then $|C| = \binom{n}{k}$ and for any distinct $u, v \in C$, we observe that $ \inprod{u}{v} \in \{0, 1, \ldots, k-1\}$. Since $\norm{u}^2= k$ for all $u \in C$, we thus obtain that $C / \sqrt{k}$ is a $\{0, 1/k, \ldots, (k-1)/k\}$-code.
\end{proof}

The following lemma follows from the well known bound for the area of a spherical cap, which can be found in \cite[Corollary 2.2]{MS86}.

\begin{lemma} \label{l random rotation}
Let $u,u' \in \R^{n}$ be unit vectors chosen independently and uniformly at random. Then for all $t> 0$, 
\[ \Pr [ \inprod{u}{u'} \geq t ] < e^{- t^2 n / 2 }.\]
\end{lemma}

We are now ready to prove \Cref{t construction}.

\begin{proof}[Proof of \Cref{t construction}]
Let $L = \{0, 1/k, \ldots, (k-1)/k\}$ and let $C$ be an $L$-code of size $\binom{n}{k}$, as given by \Cref{l k-code}. Let $\lambda = \sqrt{1/\alpha_1 - 1}$ and define 
\[ \alpha_i = \frac{\lambda^2 (i-1)/k + 1}{\lambda^2 + 1}\]
for $2 \leq i \leq k$. Note that by the choice of $\lambda$, the above also holds for $\alpha_1$. Let $L' = \{\alpha_1, \ldots, \alpha_k\}$.

Let $S$ be a set of $r+1$ unit vectors in $\mathbb{R}^r$ so that $\inprod{v}{v'} = - 1/r$ for all distinct $v,v' \in S$, i.e. $S$ is a regular $r$-simplex. For each $v \in S$, let $C_v$ be an independent and uniformly random rotation of $C$ in $\mathbb{R}^\dim$. We define
\[ C' = \left \{ \frac{(\lambda u , v)}{\sqrt{\lambda^2 + 1}}  : v \in S, u \in C_v \right \}, \]
and observe that $\norm{(\lambda u , v)}^2 = \lambda^2 + 1$, so that $C'$ is indeed a set of unit vectors in $\R^{n+r}$ of size $(1+ r) \binom{n}{k}$. Moreover, for any $v \in S$ and distinct $u, u' \in C_v$, we have 
\[ \inprod{ \frac{( \lambda u , v ) }{\sqrt{\lambda^2 + 1}} }{ \frac{( \lambda u'  , v )}{\sqrt{\lambda^2 + 1}} } = \frac{\lambda^2 \inprod{u}{u'} + 1} {\lambda^2 + 1} \in L',\]
since $\inprod{u}{u'} \in L$. Finally, suppose that $u \in C_v$ and $u' \in C_{v'}$ for distinct $v, v' \in S$. Observe that $u, u'$ are independent and uniformly random unit vectors in $\R^n$, so we may apply \Cref{l random rotation} with $t = \sqrt{\left(4\log{\binom{n}{k}} +2\log n \right)/n}$ to obtain $ \Pr [\inprod{u}{u'} \geq t] < e^{-t^2 n / 2} = n^{-1} \binom{n}{k}^{-2}$. Now define $\beta = \frac{1/r - \lambda^2 t}{\lambda^2 + 1} = \alpha_1 / r - O(\sqrt{\log(n) / n})$  and observe that if $\inprod{u}{u'} < t$, then
\[ \inprod{ \frac{ (\lambda u, v) }{\sqrt{\lambda^2 + 1}} }{ \frac{ ( \lambda u', v' ) }{\sqrt{\lambda^2 + 1}} }  = \frac{\lambda^2 \inprod{u}{u'}  + \inprod{v}{v'} }{\lambda^2 + 1} \leq \frac{\lambda^2 t - 1/r}{\lambda^2 + 1} = - \beta.\]
Thus it suffices to show that with positive probability, $\inprod{u}{u'} < t$ for all possible $u, u'$, since then $C'$ will be a $[-1, -\beta] \cup L'$-code. To that end, we observe that there are $\binom{|S|}{2} |C|^2 \leq n \binom{n}{k}^2$ such pairs $u, u'$ and so the result follows via a union bound.
\end{proof}

\section{Lines with many angles and related spherical codes}
\label{s multi-angular spherical codes}
\subsection{A general bound when $\beta$ is fixed}

In this section we give a proof of \Cref{c beta fixed}, i.e. the first statement of \Cref{t spherical}. To this end, we need a well-known variant of Ramsey's theorem, whose short proof we include for the convenience of the reader. Let $K_n$ denote the complete graph on $n$ vertices. Given an edge-colouring of $K_n$, we call an ordered pair $(X,Y)$ of disjoint subsets of vertices monochromatic if all edges in $X \cup Y$ incident to a vertex in $Y$ have the same colour. For the graph of a spherical code, we analogously call $(X,Y)$ a monochromatic $\gamma$-pair if all edges incident to a vertex in $Y$ have value $\gamma$.

\begin{lemma}
\label{l ramsey} 
  Let $k, t, m, n$ be positive integers satisfying $n>k^{kt}m$ and let $f \colon E(K_n) \to [k]$ be an edge $k$-colouring of $K_n$. Then there is a monochromatic pair $(X,Y)$ such that $|X|=m$ and $|Y|=t$.
\end{lemma}

\begin{proof}
  Construct $kt$ vertices $v_1,\dots,v_{kt}$ and sets $X_1, \ldots, X_{kt}$ as follows. Fix $v_1$ arbitrarily and let $c(1) \in [k]$ be a majority colour among the edges $(v_1,u)$. Set $X_1=\{u: f(v_1,u)=c(1)\}$. By the pigeonhole principle, $|X_1| \geq \lceil(n-1)/k\rceil \geq k^{kt-1}m$. In general, we fix any $v_{i+1}$ in $X_i$, let $c(i+1) \in [k]$ be a majority colour among the edges $(v_{i+1},u)$ with $u \in X_i$, and let $X_{i+1}=\{u \in X_i: f(v_{i+1},u)=c(i+1)\}$. Then  $|X_{i+1}| \geq \lceil(|X_i|-1)/k\rceil \geq k^{kt-i-1}m$, and for every $1 \leq j \leq i$ the edges from $v_j$ to all vertices in $X_{i+1}$ have colour $c(j)$. Since we have only $k$ colours, there is a colour $c \in [k]$ and $S \subset [kt]$ with $|S|=t$ so that $c(j) = c$ for all $j \in S$. Then $Y=\{v_j: j \in S\}$ and $X=X_{kt}$ form a monochromatic pair of colour $c$, satisfying the assertion of the lemma.
\end{proof}

We will also need the following simple corollary of Tur\'an's theorem, which can be obtained by greedily deleting vertices together with their neighbourhoods.

\begin{lemma} \label{l turan}
Every graph on $n$ vertices with maximum degree $\Delta$ contains an independent set of size at least $\frac{n}{\Delta+1}$.
\end{lemma}

Finally, we will need the bound on the size of an $L$-code previously mentioned in the introduction, see \cite{DGS77, K76}.

\begin{lemma} \label{l finite bound}
If $L \subseteq \R$ with $|L|=k$ and $C$ is an $L$-spherical code in $\R^n$ then $|C| \le \binom{n+k}{k}$.
\end{lemma}

Now we have the tools necessary to verify \Cref{c beta fixed}.

\begin{proof}[Proof of first part of \Cref{t spherical}]
We argue by induction on $k$. The base case is $k=0$, when $L =  [-1,-\beta]$, and we can take $c_{\beta,0} = \beta^{-1}+1$ by \Cref{l bounding size of a negative clique}. Henceforth we suppose $k>0$. We can assume $n \geq n_0= (2k)^{2k\beta^{-1}}$. Indeed, if we can prove the theorem under this assumption, then for $n<n_0$ we can use the upper bound for $\R^{n_0}$ (since it contains $\R^n$). Then we can deduce the bound for the general case by multiplying $c_{\beta, k}$ (obtained for the case $n \geq n_0$) by a factor $n_0^k=(2k)^{2k^2 \beta^{-1}}$. Now suppose $C$ is an $L$-code in $\R^n$, where $L =  [-1,-\beta] \cup \{\alpha_1,\dots,\alpha_k\}$, with $\alpha_1<\dots<\alpha_k$. 

Consider the case $\alpha_k < \beta^2/2$. We claim that $\Delta_{\beta} \leq 2\beta^{-2}+1$. Indeed, for any $y, x_1, x_2 \in G_C$ with $\inprod{y}{x_1}, \inprod{y}{x_2} \leq - \beta$, we have by the proof of \Cref{l angle after projecting} that
\[ \inprod{p_{y}(x_1)}{p_{y}(x_2)}
= \frac{\inprod{x_1}{x_2} - \inprod{y}{x_1} \inprod{y}{x_2} }{\sqrt{1- \inprod{y}{x_1}^2}\sqrt{1-\inprod{y}{x_2}^2}} 
\le \frac{\alpha_k - \beta^2}{\sqrt{1- \inprod{y}{x_1}^2}\sqrt{1-\inprod{y}{x_2}^2}} < -\beta^2/2\,.\]
Thus the projection of the $\beta$-neighborhood of $y$ satisfies $|p_{y}(N_{\beta}(y))| \le 2\beta^{-2}+1$ by \Cref{l bounding size of a negative clique}, as claimed. By \Cref{l turan}, the graph of $\beta$-edges in $G_C$ has an independent set $S$ of size $|C|/(2\beta^{-2}+2)$. Therefore $S$ is an $\{\alpha_1,\dots,\alpha_k\}$-code, so $|S| \le n^k+1 \leq 2n^k$ by \Cref{l finite bound}. Choosing $c_{\beta, k} > 4\beta^{-2}+4$, we see that the theorem holds in this case. Henceforth we suppose $\alpha_k \ge \beta^2/2$.

Next consider the case that there is $\ell\geq 2$ such that $\alpha_{\ell-1} < \alpha_\ell^2/2$. Choosing the maximum such $\ell$ we have 
\begin{equation}
\label{case2}
\alpha_\ell^2/2=2(\alpha_\ell/2)^2 \geq 2(\alpha_{\ell+1}/2)^4 \geq \ldots \geq 2(\alpha_k/2)^{2^{k-\ell+1}} \ge \beta' := (\beta/2)^{2^k}.
\end{equation}
Note that by induction the graph of $\{\beta, \alpha_1, \ldots, \alpha_{\ell - 1}\}$-edges in $G_C$ contains no clique of order $c_{\beta, \ell - 1} n^{\ell - 1}$, so by \Cref{l turan} its complement has maximum degree at least $m = |C| / (2 c_{\beta, \ell - 1} n^{\ell - 1})$. Letting $y \in G_C$ be a vertex attaining this maximum degree in $\{\alpha_\ell, \ldots, \alpha_k\}$-edges, we have by the pigeonhole principle that there exists $J \subset G_C$ of size at least $m / k$, and an index $\ell \leq s \leq k$ such that $\inprod{x}{y} = \alpha_s$ for all $x \in J$. Now observe that for any $x_1, x_2 \in G_C$ with $\inprod{x_1}{x_2} \in [-1, -\beta] \cup \{\alpha_1, \ldots, \alpha_{\ell-1}\}$, we have by \Cref{l angle after projecting} and \Cref{r angles go down after projecting} that
\[ \inprod{p_{y}(x_1)}{p_{y}(x_2)} =\frac{ \inprod{x_1}{x_2} - \alpha_s^2 }{1-\alpha_s^2} \leq \alpha_\ell^2/2 - \alpha_\ell^2 
< -\alpha_\ell^2/2 \le - \beta'.\]
Furthermore, by \Cref{l angle after projecting} we have that $p_y(J)$ is an $L'$-code, where $L' = [-1,-\beta'] \cup \{\alpha'_\ell,\dots,\alpha'_k\}$, with $\alpha'_i = \frac{\alpha_i - \alpha_s^2}{1-\alpha_s^2}$ for $i \ge \ell$. By the induction hypothesis, we have \\
$|J| \le c_{\beta', k - \ell + 1} n^{k-\ell+1}$, so choosing $c_{\beta, k} > 2 k c_{\beta, \ell - 1} c_{\beta', k - \ell + 1}$ the theorem holds in this case.

Now suppose that there is no $\ell>1$ such that $\alpha_{\ell-1} < \alpha_\ell^2/2$. We must have $\alpha_1>0$. Let $t=\lceil 1/\beta' \rceil$.
We apply \Cref{l ramsey} to find a monochromatic pair $(X,Y)$ with $|Y|=t$ and $|X|=m\geq  (k+1)^{-(k+1)t} n$. Since $G_C$ has no $\beta$-clique of size $t$ by \Cref{l bounding size of a negative clique}, $(X,Y)$ must be a monochromatic $\alpha_r$-pair for some $1 \leq r \leq k$. Let $X' = p_Y(X)$ be the projection of $X$ onto the orthogonal complement of $Y$. By \Cref{l angle after projecting} and \Cref{r angles go down after projecting}, we have that $X'$ is a $[-1, \beta] \cup \{\alpha_1', \ldots, \alpha_k'\}$-code, where 
\[ \alpha_i' = \frac{\alpha_i - \alpha_r}{1 - \alpha_r} + \frac{\alpha_r(1 - \alpha_i)}{(1 + \alpha_r t)(1 - \alpha_r)} \quad \text{for } 1 \leq i \leq k. \]
We can assume $\alpha'_k \ge \beta^2/2$, otherwise choosing $c_{\beta, k} > (k+1)^{(k+1)t} (4\beta^{-2}+4)$ we are done by the first case considered above. Since $\alpha'_r=(\alpha_r^{-1}+t)^{-1}<\beta'$, the computation in (\ref{case2}) implies that there exists $\ell>1$ such that $\alpha_{\ell-1} < \alpha_\ell^2/2$. Choosing $c_{\beta, k} > (k+1)^{(k+1)t} 2 k c_{\beta, \ell - 1} c_{\beta', k - \ell + 1}$ we are done by the second case considered above.
\end{proof}

\subsection{An asymptotically tight bound when $\beta, \alpha_1, \ldots, \alpha_k$ are fixed}

The goal of this section will be to prove the second statement of \Cref{t spherical}. The case $k = 1$ is given by \Cref{t spherical 1}, so henceforth we assume that we are given a fixed $k \geq 2$. 

Our general strategy will be to use projections in order to reduce the number of positive angles and then apply induction. When projecting onto the orthogonal complement of a large clique, \Cref{l angle after projecting} tells us that the new inner product will be some function of the old one plus $o(1)$. In view of this, it will be convenient to prove the following, slightly more general version of \Cref{t spherical}.

\begin{theorem} \label{t spherical induction}
Let $\beta \in (0,1]$, $\alpha_1, \ldots, \alpha_k \in [0, 1)$ be fixed with $\alpha_1 < \ldots < \alpha_k$ and let $\dim \in \N$. If $C$ is a spherical $[-1, -\beta + o(1)] \cup \{ \alpha_1 + o(1), \ldots, \alpha_k + o(1)\}$-code in $\R^n$, then
\[ |C| \leq \left( 1 + \frac{\alpha_1}{\beta} \right)(k-1)!(2n)^k + o(n^k)\]
for $n$ sufficiently large.
\end{theorem}

\begin{rem}
We use $\gamma + o(1)$ to refer to a specific $\gamma^* \in \R$ depending on $n$ such that $\gamma^* = \gamma + o(1)$, not a range of possible values near $\gamma$. We will still say $\gamma$-edge, $\Delta_\gamma$, etc. when we are referring to a $(\gamma + o(1))$-edge, $\Delta_{\gamma + o(1)}$, etc. in the graph $G_C$.
\end{rem}

The case $k = 1$ of \Cref{t spherical induction} follows by inserting $o(1)$ terms into expressions in the proof of \Cref{t spherical 1}, and so we henceforth assume that it holds. Moreover, since we will be making use of induction, we also assume that \Cref{t spherical induction} holds for all $k' < k$. Now let $\beta \in (0,1]$ and $\alpha_1, \ldots, \alpha_k \in [0, 1)$ be fixed with $\alpha_1 < \ldots < \alpha_k$ and let $n \in \N$. Let $C$ be a spherical $[-1, -\beta + o(1)] \cup \{ \alpha_1 + o(1), \ldots, \alpha_k + o(1)\}$-code in $\R^n$. 

The argument will be a generalisation of the one used to prove \Cref{t spherical 1} and so we will need to generalise some lemmas. Firstly, we will need the following generalisation of \Cref{l one angle garbage}.

\begin{lemma} \label{l garbage}
Let $\alpha \in (0,1)$ and $\gamma \in [-1,1)$ be distinct reals. Let $X \cup Y \cup \{ z\}$ be a set of unit vectors in $\mathbb{R}^\dim$ so that $Y \cup \{z\}$ is an $\{\alpha\}$-code, that all edges inside $X$ have value at most $\alpha$, that all edges between $X$ and $Y$ have value at least $\alpha$ and that all edges between $X$ and $z$ all have value at least $\gamma$ if $\gamma > \alpha$ and at most $\gamma$ if $\gamma < \alpha$. Suppose furthermore that $|Y| > 4/\bigl(\alpha(\gamma-\alpha)^2\bigr)$. Then $|X| < 1/(\gamma - \alpha)^2$.
\end{lemma}

\begin{proof}
  Let $\alpha_X$ denote the average value of the edges in side $X$, $\alpha_Y$ the average value of the edges between $X$ and $Y$ and $\gamma_z$ the average value of the edges between $X$ and $z$. Note that our assumptions imply $\alpha_Y \ge \alpha \ge \alpha_X$ and $|\gamma_z - \alpha| \ge |\gamma - \alpha|$. Let $M$ be the Gram matrix of $X \cup Y \cup \{z \}$ and let $ v = (x, \dots, x, y, \dots, y, \zeta)^\intercal$, where 
  \begin{equation*}
    x = 1/|X|, \quad y = -\frac{(\alpha_Y-\alpha \gamma_z)/|Y|}{\alpha - \alpha^2 + (1 -\alpha)/|Y|} \quad \text{and} \quad \zeta = -(\gamma_z + y |Y| \alpha).
  \end{equation*}
  Then, similarly to the proof of \Cref{l one angle garbage},
  \begin{equation*}
  \begin{split}
     \inprod{Mv}{v} &= \frac{|X|+(|X|^2-|X|)\alpha_X}{|X|^2} + 2y|Y|\alpha_Y + y^2((|Y|^2-|Y|)\alpha + |Y|) + 2\zeta(\gamma_z + y|Y|\alpha) + \zeta^2 \\
     &= \frac{1-\alpha_X}{|X|}+\alpha_X -\gamma_z^2 + 2y|Y|(\alpha_Y- \alpha \gamma_z) + y^2|Y|^2\bigl(\alpha -\alpha^2 + (1-\alpha)/|Y|\bigr) \\
     &\le \frac{1-\alpha}{|X|}+\alpha -\gamma_z^2 - \frac{(\alpha_Y - \alpha\gamma_z)^2}{\alpha -\alpha^2 + (1-\alpha)/|Y|},
  \end{split}
  \end{equation*}
  where the last inequality uses $\alpha_X \le \alpha$. Note that $\alpha_Y \ge \alpha \ge \alpha \gamma_z$, so $(\alpha_Y-\alpha\gamma_z)^2 \ge \alpha^2(1-\gamma_z)^2$. Since $ \inprod{Mv}{v} \ge 0$ and $(\gamma_z - \alpha)^2 \ge (\gamma - \alpha)^2$, it is therefore sufficient to prove that
  \begin{equation*}
    \alpha -\gamma_z^2 - \frac{\alpha^2(1 - \gamma_z)^2}{\alpha -\alpha^2 + (1-\alpha)/|Y|} < -(\gamma_z - \alpha)^2. 
  \end{equation*}
  As one can easily check, this can be rewritten as $|Y| \ge (1-\alpha)(1+\alpha - 2\gamma_z)/\alpha(\gamma_z-\alpha)^2$, which is true by assumption.
\end{proof}

We will also need the following generalisation of \Cref{l beta edges}.

\begin{lemma}
  \label{l degree bound}
  If $\alpha_1 = 0$, then $G_C$ has the following degree bounds:
  \begin{enumerate}
    \item[(i)] $\Delta_{\alpha_i} \leq \frac{1}{\alpha_i} (k-2)! (2n)^{k-1} + o(n^{k-1})$ for all $2 \leq i \leq k$.
    \item[(ii)] If $-\beta_1, \dots , -\beta_{N}$ are the values of the $\beta$-edges incident to some $x \in G_C$, then $N \leq O(n^{k-1})$ and
      \begin{align*}
        \sum_{i = 1}^{N} \beta_i^2 \le (k-1)! (2n)^{k-1} + o(n^{k-1}).
      \end{align*}
  \end{enumerate}
\end{lemma}

\begin{proof}
Let $x \in C$ and let $C' = p_{x}(N_{\alpha_i}(x))$ be the normalised projection of the $\alpha_i$-neighbours of $x$ onto $\Span(x)^{\perp}$. By \Cref{l angle after projecting}, we see that $C'$ is a $[-1, -\beta' + o(1)] \cup \{\alpha_1' + o(1), \ldots, \alpha_k' + o(1)\}$-spherical code for
\begin{equation*}
\alpha_j' = \frac{\alpha_j - \alpha_i^2}{1 - \alpha_i^2} \text{ for } 1 \leq j \leq k, \quad 
\beta' = \frac{\beta + \alpha_i^2}{1 - \alpha_i^2}.
\end{equation*}
In particular $\alpha_{1}' = \frac{ - \alpha_i^2}{1 - \alpha_i^2} < 0$. Now let $\ell$ be the largest integer such that $\alpha_{\ell}' < 0$ and observe that $C'$ is, in particular, a $[-1, \alpha_{\ell}' + o(1)] \cup \{\alpha_{\ell + 1}' + o(1), \ldots, \alpha_k' + o(1)\}$-code. If $\ell \geq 2$ then applying \Cref{t spherical induction} by induction we obtain $|C'| \leq O(n^{k-\ell})$, which trivially implies $(i)$. Otherwise $\alpha_2' \geq 0$ and applying \Cref{t spherical induction} by induction we obtain
\begin{equation*}
|C'| \leq \left(1 + \frac{\alpha'_{2}}{-\alpha_{1}'} \right) (k-2)! (2n)^{k-1} + o(n^{k-1}).
\end{equation*}
To verify $(i)$, it suffices to observe that $1 - \alpha_2 '/ \alpha_1' = 1 + (\alpha_2 - \alpha_i^2) / \alpha_i^2 = \alpha_2 / \alpha_i^2 \leq 1 / \alpha_i$.

Now we derive the upper bound on $N = d_{\beta}(x)$. Let $M = M_{N_{\beta}(x) \cup \{x\}}$ be the Gram matrix of $N_{\beta}(x) \cup \{x\}$ and let $v = (1, \ldots, 1, \beta N)^{\intercal}$. Then using $(i)$ we conclude
\begin{equation*}
\begin{split}
  0 \le v^\intercal M v 
  &\leq \sum_{1\le i,j \le N} M_{ij} - \left(\beta^2 + o(1) \right) N^2 \\
  &\leq N \left(1 + (\alpha_{1} + o(1)) N + \sum_{j=2}^k{(\alpha_j + o(1))\Delta_{\alpha_j}} \right) - \left(\beta^2 + o(1) \right) N^2\\
  &\leq N \left(1 + o(1) N + \sum_{j=2}^k{(\alpha_j + o(1) ) \left(\frac{1}{\alpha_j} (k - 2)! (2n)^{k-1} + o(n^{k-1}) \right) } \right) - \left(\beta^2 + o(1) \right) N^2\\
  &\leq N \left( 1 + (k-1)! (2n)^{k-1}  + o(n^{k-1}) \right) - \left(\beta^2 + o(1) \right) N^2,
\end{split}
\end{equation*}
which implies $N \leq  \frac{1}{\beta^2} (k-1)! (2n)^{k-1} + o(n^{k-1}) = O(n^{k-1})$.

Finally, let $-\beta_1, \dots , -\beta_{N}$ be the values of the $\beta$-edges incident to $x$ and let \\
$w = \left( \beta_1, \dots, \beta_N, \sum_{i=1}^N{\beta_i^2} \right)^\intercal$. Then
\begin{equation*}
\begin{split}
    0 \le w^\intercal M w 
    &= - \left( \sum_{i=1}^N{\beta_i^2} \right)^2 + \sum_{i=1}^N{\beta_i^2}  + \sum_{\substack{1\le i,j \le N \\ i \ne j}} \beta_i\beta_j M_{ij}\\
    &\leq - \left(  \sum_{i=1}^N{\beta_i^2} \right)^2 +  \sum_{i=1}^N{\beta_i^2} + \sum_{r = 2}^k{ \alpha_r \Bigg(\sum_{\substack{i,j \\ M_{i,j} = \alpha_r + o(1) }}{\beta_i \beta_j} \Bigg) } +  o(1)\sum_{1 \leq i,j \leq N}{\beta_i \beta_j}.
\end{split}
\end{equation*}
Applying Cauchy--Schwarz, we obtain
\begin{equation*}
\sum_{1 \leq i,j \leq N}{\beta_i \beta_j} 
= \left( \sum_{i=1}^N{\beta_i} \right)^2 
\leq N \sum_{i=1}^{N}{\beta_i^2}\\
\leq O(n^{k-1}) \sum_{i=1}^{N}{\beta_i^2}.
\end{equation*}
Furthermore, for $2 \leq r \leq k$ we have
\begin{equation*}
0 \leq \frac{1}{2} \sum_{\substack{i,j \\ A_{i,j} = \alpha_r + o(1)}}{(\beta_i - \beta_j)^2} 
\leq \Delta_{\alpha_r} \sum_{i=1}^{N}{\beta_i^2} - \sum_{\substack{i,j \\ A_{i,j} = \alpha_r + o(1) }}{\beta_i \beta_j},
\end{equation*}
and thus we obtain
\begin{equation*}
\alpha_r \sum_{\substack{i,j \\ A_{i,j} = \alpha_r + o(1)}}{\beta_i \beta_j} 
\leq \left( (k-2)! (2n)^{k-1} + o(n^{k-1}) \right) \sum_{i=1}^{N}{\beta_i^2}.
\end{equation*}
Combining these inequalities and dividing by $\sum_{i=1}^N{\beta_i^2}$ yields the desired $(ii)$:
\[  \sum_{i = 1}^{N} \beta_i^2 \le (k-1)! (2n)^{k-1} + o(n^{k-1}). \qedhere \]
\end{proof}

Finally, we will need a new lemma to deal with what happens if the clique we find via Ramsey's theorem is an $\alpha_i$-clique for $i \geq 2$.

\begin{lemma} \label{l alpha_2}
Let $2 \leq i \leq k$ and suppose $X \cup Y$ is a $[-1, -\beta] \cup \{ \alpha_1, \ldots, \alpha_k\}$-spherical code with $|Y| \rightarrow \infty$ as $n \rightarrow \infty$, such that all edges incident to any $y \in Y$ are $\alpha_i$-edges. Then $|X| \leq O(n^{k-1})$.
\end{lemma}

\begin{proof}
  Let $X' = p_Y(X)$ be the normalised projection of $X$ onto $\Span(Y)^\perp$. By \Cref{l angle after projecting}, we have that $X'$ is a $[-1, -\beta' + o(1)] \cup \{ \alpha_1' + o(1), \ldots, \alpha_k' + o(1)\}$-code for 
  \begin{equation*}
\alpha_j' = \frac{\alpha_{j} - \alpha_i }{1-\alpha_i} \text{ for } 1 \leq j \leq k, \quad 
\beta' = \frac{\beta + \alpha_i}{1-\alpha_i}.
\end{equation*}
  Observe that $\alpha_{i-1}' = (\alpha_{i-1} - \alpha_i) / (1 - \alpha_i)  < 0$, so that $X'$ is, in particular, a $[-1, \alpha_{i-1}' + o(1)] \cup \{\alpha_{i}' + o(1), \ldots, \alpha_k' + o(1) \}$-code and hence we may apply \Cref{t spherical induction} by induction to conclude $|X'| \leq O(n^{k-i+1}) \leq O(n^{k-1})$.
\end{proof}

We now have all of the necessary lemmas to finish the proof of \Cref{t spherical induction}.

\begin{proof}[Proof of \Cref{t spherical induction}]
Suppose first that $\alpha_1 = 0$. Let $Q = M_C - (\alpha_1 + o(1)) J$; by the subadditivity of the rank we have $\rank{Q} \leq \rank{M_C} + \rank{J} \le n + 1$. Now fix some $x \in C$, and let $N = d_{\beta}(x)$ and $\beta_1, \ldots, \beta_N$ be the values of the $\beta$-edges incident to $x$. Using parts (i) and (ii) of \Cref{l degree bound}, it follows that if $i$ is the row corresponding to $x$ in $Q$ then
\begin{equation*}
\begin{split}
\sum_{j \neq i}{Q_{i,j}^2} 
&\leq \sum_{j=1}^N{(-\beta_j - (\alpha_1 + o(1)))^2} + \sum_{r = 2}^{k}{(\alpha_r - \alpha_1 + o(1))^2 \Delta_{\alpha_r} }\\
&\leq (1 + o(1) )\left( (k-1)! (2n)^{k-1} + o(n^{k-1}) \right) + \sum_{r = 2}^{k}{ (k-2)! (2n)^{k-1} + o(n^{k-1}) } \\
&\leq 2(k-1)! (2n)^{k - 1} + o(n^{k - 1}).
\end{split}
\end{equation*}
Noting that $Q$ has $ 1 - (\alpha_1 + o(1)) = 1 - o(1)$ on the diagonal, we obtain
\[ \tr(Q^2) = \sum_{i = 1}^{|C|}{Q_{i,i}^2} + \sum_{i = 1}^{|C|}{\sum_{j \neq i}{ Q_{i,j}^2 }} \leq |C| \left( 1 +  2(k-1)! (2n)^{k - 1} + o(n^{k - 1}) \right). \]
Thus applying \Cref{l schnirelman trick} to $Q$ yields
\[ |C|^2(1-o(1))^2 = \tr(Q)^2 \leq \tr(Q^2) \rank{Q} \leq |C| \left( 1 + 2(k-1)! (2n)^{k - 1} + o(n^{k - 1}) \right) (n+1). \]
Dividing by $|C|(1-o(1))^2$, we obtain the required $|C| \leq (k-1)! (2n)^k + o(n^k)$.

We will now prove the theorem for $\alpha_1 > 0$. Let $t = \log{\log{n}}$, let $\epsilon \rightarrow 0$ sufficiently slowly as $n \rightarrow \infty$ and suppose for sake of contradiction that $|C| \geq (1 + \alpha_1 / \beta)(k-1)! (2n)^{k} + \epsilon n^{k}$. Let $m = \lceil |C|/(k+1)^{(k+1)t} - 1 \rceil$ so that $|C| > (k+1)^{(k+1)t}m$.
  
  Regarding a $\gamma$-edge of $C$ as an edge coloured with the colour $\gamma$, we deduce by \Cref{l ramsey} that there are some subsets $X$ and $Y$ of $C$ so that $|X| = m$, $|Y| = t$ and $(X, Y)$ is a monochromatic $\gamma$-pair for some $\gamma \in \{ \beta, \alpha_1, \ldots, \alpha_k\}$. Since $t > \frac{1}{\beta} +1$ for $n$ sufficiently large, $Y$ cannot be a $\beta$-clique by \Cref{l bounding size of a negative clique} and hence $(X,Y)$ cannot be a monochromatic $\beta$-pair. Hence it must be a monochromatic $\alpha_i$-pair. If $2 \leq i \leq k$, then by \Cref{l alpha_2} we conclude $\Omega( n^k / (k+1)^{(k+1)t } ) \leq m \leq O(n^{k-1})$, a contradiction for $n$ large enough by our choice of $t$. Hence, $(X,Y)$ must be a monochromatic $\alpha_1$-pair and hence $C$ contains an $\alpha_1$-clique $Y$ of size $t$.
  
  For each $T \subseteq Y$, let $S_T$ be the set of vertices $x \in G_C \setminus Y$ so that $N_{\beta}(x) \cap Y = Y \setminus T$. Now fix $T \subseteq Y$ and let $t_1, \dots, t_{|T|}$ be some ordering of the elements of $T$. For each pattern of the form $p \in [k]^{|T|}$, let $S_T(p)$ consist of all $x \in S_T$ for which $x \cdot t_i = \alpha_{p_i} + o(1)$ for all $i$.
  
  Define $t^* = 4/\bigl(\alpha_1(\beta + \alpha_1)^2\bigr) + o(1)$ and suppose first that $t^* \leq |T| < t$. We claim that $S_T(p)$ does not contain an $\alpha_1$-clique of size larger than $1 / \beta^2 +  o(1)$ for any $p \in [k]^{|T|}$. To that end, fix some $z \in Y \setminus T $ and let $X$ be an $\alpha_1$-clique in $S_T(p)$. Note that, for any $x \in X$, $\inprod{x}{z} < -\beta + o(1)$ and $|T| \ge t^* > 4/\bigl((\alpha+o(1))(\beta+\alpha+o(1))^2\bigr)$, so that we may apply \Cref{l garbage} to $X \cup T \cup \{z\}$ to conclude that $|X| < 1/(\beta+\alpha)^2 + o(1) < 1/ \beta^2 + o(1)$. 
  
  Now let $m' = \lceil |S_T(p)|/(k+1)^{(k+1)t} - 1 \rceil$ so that $|S_T(p)| > (k+1)^{(k+1)t}m'$. Then by \Cref{l ramsey}, $S_T(p)$ contains an $(X',Y')$ monochromatic pair with $|X'| = m'$ and $|Y'| = t$. Since $t > 1 / \beta^2 + o(1)$ for $n$ large enough, $Y'$ cannot be a monochromatic $\alpha_1$-clique or $\beta$-clique, and thus $(X',Y')$ is a monochromatic $\alpha_i$-pair for some $2 \leq i \leq k$. Thus by \Cref{l alpha_2}, we conclude that $m' \leq O(n^{k-1})$. Since this holds for all $p$, we obtain
  \[ |S_T| =  \sum_{p \in [k]^{|T|}}{|S_T(p)|} \leq k^{|T|} (k+1)^{(k+1)t} O \left( n^{k-1} \right) \leq (k+1)^{(k+2)t} O\left( n^{k-1} \right).\]

Now suppose that $T = Y$ and let $p \in [k]^{t} \backslash \{(1, \ldots, 1)\}$. We claim that $S_Y(p)$ does not contain an $\alpha_1$-clique of size larger than $1 / (\alpha_2 - \alpha_1)^2 + o(1)$. To that end, fix an index $j$ such that $p_j \ge 2$ and let $X$ be an $\alpha_1$-clique in $S_Y(p)$. Note that, for any $x\in X$, $\inprod{t_j}{x} = \alpha_{p_j}+o(1) \ge \alpha_2+o(1)$. Furthermore, for sufficiently large $\dim$, we have $t > 4 /\bigl(\alpha_1(\alpha_2 - \alpha_1 )^2\bigr) + o(1)$. Therefore, we may apply \Cref{l garbage}, with $\alpha = \alpha_1 + o(1)$ and $\gamma = \alpha_2 + o(1)$, to $X \cup T' \cup \{z\}$, where $z = t_j$ and $T' = Y \backslash \{t_j\}$, to conclude that $|X| < 1/(\alpha_2 - \alpha_1)^2 + o(1)$. 

As above, let $m' = \lceil |S_Y(p)|/(k+1)^{(k+1)t} - 1 \rceil$ and observe that by \Cref{l ramsey}, $S_Y(p)$ contains an $(X',Y')$ monochromatic pair with $|X'| = m$ and $|Y| = t$. Since $t$ is large enough, it cannot be a $\beta$ or $\alpha_1$-pair, so it must be a monochromatic $\alpha_i$ for some $2 \leq i \leq k$, and hence by \Cref{l alpha_2}, we conclude that $m' \leq O(n^{k-1})$. Thus we obtain
\[ \sum_{p \in [k]^{t} \backslash \{(1, \ldots, 1)\}}{|S_Y(p)|} \leq k^{|T|} (k+1)^{(k+1)t} O \left( n^{k-1} \right) \leq (k+1)^{(k+2)t} O\left( n^{k-1} \right) .\]

Finally, let $p = (1, \ldots, 1)$. and define $X' = p_{Y}(S_Y(1, \ldots, 1))$ to be the normalised projection of $S_Y(1,\ldots, 1)$ onto $\Span(Y)^{\perp}$. By \Cref{l angle after projecting} we have that $X'$ is a  $[ -1, - \beta' + o(1) ] \cup \{\alpha_1' + o(1), \ldots, \alpha_k' + o(1)\}$-spherical code for 
\begin{equation*}
\alpha_j' = \frac{\alpha_{j} - \alpha_1}{1-\alpha_1} \text{ for } 1 \leq j \leq k, \quad 
\beta' = \frac{\beta + \alpha_1}{1-\alpha_1}.
\end{equation*}
Since $\alpha_1' = 0$, we can apply the previous case of \Cref{t spherical induction} to obtain $|X'| \leq (k-1)! (2n)^k + o(n^k)$. It follows that 
\begin{equation*}
\begin{split} 
|S_Y| = |S_Y(1, \ldots, 1)| + \sum_{p \in [k]^{t} \backslash \{(1, \ldots, 1)\}}{|S_Y(p)|} &\leq (k-1)! (2n)^k + o(n^k) + (k+1)^{(k+2)t} O \left( n^{k-1} \right)\\
&= (k-1)! (2n)^k + o(n^k).
\end{split}
\end{equation*}
Noting that $(k+1)^{(k+2)t} = o(n)$, we therefore obtain
 \begin{equation*} 
 \begin{split}
 \left | \bigcup_{T \subseteq Y, |T| \geq t^*}{S_T} \right | 
 = |S_Y| + \left | \bigcup_{T \subset Y, |T| \geq t^*}{S_T} \right |
 &\leq (k-1)! (2n)^k + o(n^k) +  2^t(k+1)^{(k+2)t} O \left( n^{k-1} \right) \\
 &= (k-1)! (2n)^k + o(n^k).
 \end{split}
 \end{equation*}
 
Thus if we define $G' = G_C \backslash \left (  \bigcup_{T \subseteq Y, |T| \geq t^*}{S_T} \right )$ then $|G'| \geq (\alpha_1 / \beta)(k-1)!(2n)^k + (\epsilon - o(1))n^k$. Hence we can iterate the above procedure $\ell = \alpha_1 / \beta + 1$ times to obtain disjoint $\alpha_1$-cliques $Y_1, \ldots, Y_{\ell}$ and a disjoint graph $G'$ of size at least $(\epsilon - o(1))n^2$. By having $\epsilon \rightarrow 0$ slowly enough, we can apply \Cref{l ramsey} one more time to $G'$ to obtain a monochromatic $\alpha_1$-pair, which gives an additional $\alpha_1$-clique $Y_{\ell +1} \subseteq G'$ of size $t$. Note that by construction, the number of $\beta$-edges between $Y_i$ and $Y_j$ is at least $t(t - t^*) = t^2 ( 1 - o(1) )$ for distinct $i$ and $j$. But then we can apply \Cref{l multipartite graphs} to obtain $\ell + 1 \leq \alpha_1 / \beta + 1 + o(1)$, a contradiction for $n$ large enough.
\end{proof}

\section{Concluding remarks}
\label{s concluding remarks}

In this paper, we showed that the maximum cardinality of an equiangular set of lines with common angle $\arccos \alpha$ is at most $2\dim -2$ for fixed $\alpha \in (0,1)$ and large $\dim$. Moreover, we proved that this bound is only attained for $\alpha = \frac{1}{3}$ and that we have an upper bound of $\finalconstantplusone \dim$ otherwise. In view of the result of Neumann \cite[p. 498]{LS73}, it is not too surprising that $\limsup_{n \rightarrow \infty}{N_{\alpha}(n) / n}$ should be biggest when $1/\alpha$ is an odd integer. What is surprising, however, is that a maximum occurs at all and moreover that it happens when $\alpha$ is large. Indeed, the constructions of $\Omega(n^2)$ equiangular lines have $\alpha \rightarrow 0$, and so one might a priori expect that $\limsup_{n \rightarrow \infty}{N_{\alpha}(n) / n}$ should increase as $\alpha$ decreases.

If $\alpha = 1/(2r-1)$ for some positive integer $r$, an analogous construction as for $\alpha = \frac{1}{3}$ yields an equiangular set of $r\lfloor (\dim-1)/(r-1)\rfloor$ lines with angle $\arccos(1/(2r-1))$. Indeed, consider a matrix with $t = \lfloor (\dim-1)/(r-1)\rfloor$ blocks on the diagonal, each of size $r$, with $1$ on the diagonal and $-\alpha$ off the diagonal; all other entries are $\alpha$. One can show that this is the Gram matrix for a set of $rt$ unit vectors in $\mathbb{R}^\dim$. For $n$ large enough and $r=2$ \cite{vLS66} and $r=3$ \cite{N89}, it is known that this construction is optimal. This motivates the following conjecture, which was also raised by Bukh \cite{B15}.

\begin{conj}
  Let $r\ge 2$ be a positive integer. Then, for sufficiently large $\dim$, 
  \begin{equation*}
    \equimaxangle{\frac{1}{2r-1}}{\dim} = \frac{r (\dim - 1)}{r-1} + O(1).
  \end{equation*}
\end{conj}

\noindent
If $\alpha$ is not of the above form, the situation is less clear  but it is conceivable that $\equimaxangle{\alpha}{\dim} = (1+o(1))\dim$. 

We believe that the tools developed here should be useful to determine the asymptotics of $\equimaxangle{\alpha}{\dim}$ for every fixed $\alpha$. If $\alpha$ is allowed to depend on $n$, then our methods work provided that $\alpha > \Theta(\log^{-1}{n})$. The only place where this assumption is really necessary is our use of Ramsey's theorem in order to obtain a large positive clique. However, it is conceivable that a large positive clique exists even when $\alpha < \Theta(\log^{-1}{n})$, in which case our methods would continue to be effective.

We have also proved an upper bound of $O(n^k)$ for a set of lines attaining $k$ prescribed angles. If the angles can tend to $0$ together with $\dim$, however, this bound no longer applies and the general bound of $O(n^{2k})$ by Delsarte, Goethals and Seidel \cite{DGS75} remains best possible. There are by now plentiful examples showing that for $k=1$ their bound gives the correct order of magnitude, but no such constructions are known for other values of $k$. So it would be interesting to determine whether the bound of Delsarte, Goethals and Seidel is tight for $k \geq 2$.

\vspace{0.25cm}
\noindent
{\bf Acknowledgment.}\,
We would like to thank Boris Bukh and the referee for useful comments.

\end{document}